\documentclass[preprint,3p]{elsarticle}
\usepackage{graphicx}
\usepackage{psfrag}
\usepackage{amssymb}
\usepackage{amsthm}
\usepackage{lineno}
\usepackage{amsfonts}
\usepackage{amsmath}
\usepackage{pstricks-add}
\usepackage{bm}
\usepackage{color,hyperref}
\hypersetup{colorlinks,breaklinks,
           linkcolor=blue,urlcolor=blue,
           anchorcolor=blue,citecolor=blue}
\usepackage[caption=false]{subfig}
\usepackage{float}
\usepackage{comment}
\usepackage{amsthm}
\newtheorem{thm}{Theorem}[section]
\newtheorem{lem}{Lemma}[section]

\graphicspath{{fig/}}

\journal{Unknown}
\begin{document}
\begin{frontmatter}
\title{Compressive Sampling of Polynomial Chaos Expansions: Convergence Analysis and Sampling Strategies}
\author{Jerrad Hampton}
\author{Alireza Doostan\corref{cor1}}
\ead{alireza.doostan@colorado.edu}
\cortext[cor1]{Corresponding Author: Alireza Doostan}

\address{Aerospace Engineering Sciences Department, University of Colorado, Boulder, CO 80309, USA}

\begin{abstract}
Sampling orthogonal polynomial bases via Monte Carlo is of interest for uncertainty quantification of models with high-dimensional random inputs, using Polynomial Chaos (PC) expansions. It is known that bounding a probabilistic parameter, referred to as {\it coherence}, yields a bound on the number of samples necessary to identify coefficients in a sparse PC expansion via solution to an $\ell_1$-minimization problem. {\color{black}Utilizing results} for orthogonal polynomials, we bound the coherence parameter for polynomials of Hermite and Legendre type under {\color{black}their} respective natural sampling distribution. In both polynomial bases we identify an importance sampling distribution which yields a bound with weaker dependence on the order of the approximation. For more general orthonormal bases, we propose the {\it coherence-optimal} sampling: a Markov Chain Monte Carlo sampling, which directly uses the basis functions under consideration to achieve a statistical optimality among all sampling schemes with identical support. We demonstrate these different sampling strategies numerically in both high-order and high-dimensional, manufactured PC expansions.  In addition, the quality of each sampling method is compared in the identification of solutions to two differential equations, one with a high-dimensional random input and the other with a high-order PC expansion. In both cases the coherence-optimal sampling scheme leads to similar or considerably improved accuracy.

\end{abstract}
\end{frontmatter}

\section{Introduction}
\label{sec:intro}

A precise approach to analyzing modern, sophisticated engineering systems requires understanding how various Quantities of Interest (QoI) behave as functions of uncertain system inputs. An ineffective understanding may give unfounded confidence in the QoI or suggest unnecessary restrictions in the system inputs due to unnecessary incredulity concerning the QoI. This process of Uncertainty Quantification (UQ) has received much recent study~\cite{Ghanem91a,LeMaitre10,Xiu10a}.

Probability is a natural framework for modeling uncertain inputs by assuming the input depends on a $d$-dimensional random vector $\bm{\Xi}:=(\Xi_1,\cdots,\Xi_d)$ with some joint probability density function $f(\bm{\xi})$. In this manner we model the scalar QoI, denoted by $u(\bm{\Xi})$, as an unknown function of the input, which we seek to approximate. In this work we approximate $u(\bm{\Xi})$, assumed to have a finite variance, using an expansion in multivariate orthogonal polynomials, each of which we denote by $\psi_k(\bm{\Xi})$, yielding a Polynomial Chaos (PC) expansion~\cite{Ghanem91a,Xiu02},
\begin{align}
\label{Eq:PCEDef}
u(\bm{\Xi}) &= \mathop{\sum}\limits_{k=0}^\infty c_k \psi_{k}(\bm{\Xi}),\\
\nonumber
&\approx \mathop{\sum}\limits_{k\in\mathcal{C}} c_k \psi_{k}(\bm{\Xi}).
\end{align}
Under conditions discussed in Section~\ref{subsec:PCE}, the index set $\mathcal{C}$ may have few elements, allowing us to accurately reconstruct $u$ from a relatively small number of basis polynomials, i.e., there exists a {\it sparse} representation for $u$ as a linear combination of orthogonal polynomials in $\bm{\Xi}$. For computation we truncate the expansion in (\ref{Eq:PCEDef}) so that we have $\bm{c}=(c_1,\cdots,c_P)^T$ and
\begin{align}
\label{Eq:PCETrunc}
u(\bm{\Xi}) &\approx \mathop{\sum}\limits_{k=1}^P c_k \psi_{k}(\bm{\Xi}),
\end{align}
where the error introduced by this truncation to a finite number of terms is referred to as {\it truncation error}. The polynomials $\psi_{k}(\bm{\Xi})$ are naturally selected to be orthogonal with respect to the measure $f(\bm\xi)$ of the inputs $\bm\Xi$, \cite{Xiu02,Soize05}. For instance, when $\bm\Xi$ follows a jointly uniform or Gaussian distribution (with independent components), $\psi_{k}(\bm{\Xi})$ are multivariate Legendre or Hermite polynomials, respectively. For the interest of analysis, we assume that $\psi_{k}(\bm{\Xi})$ are normalized such that $\mathbb{E}[\psi^2_{k}(\bm{\Xi})]=1$, where $\mathbb{E}$ denotes the mathematical expectation operator. If we can accurately identify the coefficients $c_k = \mathbb{E}[u(\bm\Xi) \psi_{k}(\bm\Xi)]$ for our approximation, then as $P\rightarrow\infty$ there is the mean-squares convergence of our PC approximation to $u$. 

To identify $\bm{c}$ we consider non-intrusive, i.e., sampling-based, methods where we do not require changes to deterministic solvers for $u$ as we generate realizations of $\bm{\Xi}$ to identify $u(\bm{\Xi})$. We denote these realizations $\bm{\xi}^{(i)}$ and $u(\bm{\xi}^{(i)})$, respectively. We let $i=1:N$ so that $N$ is the number of independent samples considered, and define
\begin{align}
\label{eqn:psi_u}
\bm{u}&:=(u(\bm{\xi}^{(1)}),\cdots,u(\bm{\xi}^{(N)}))^T;\\
\bm{\Psi}(i,j)&:=\psi_{j}(\bm{\xi}^{(i)}).\nonumber
\end{align}
These definitions imply the matrix equality $\bm{\Psi}\bm{c}=\bm{u}$. We also introduce a diagonal positive-definite matrix $\bm{W}$ such that $\bm{W}(i,i)$ is a function of $\bm{\xi}^{(i)}$ that depends on our sampling strategy and is described in Sections~\ref{sec:motivation} and~\ref{sec:sampling}. To approximate $\bm{c}$ we use Basis Pursuit Denoising (BPDN), \cite{Chen98,Chen01,Donoho06b,Bruckstein09}. This involves solving either the $\ell_1$-minimization problem
\begin{align}
\label{eqn:constrained}
\mathop{\arg\min}_{\bm{c}}\|\bm{c}\|_1 \mbox{ subject to } \|\bm{W}\bm{u}-\bm{W}\bm{\Psi}\bm{c}\|_2\le\delta,
\end{align}
where $\delta$ is a tolerance of solution inaccuracy due to the truncation error, or the closely related
\begin{align}
\label{eqn:regularized}
\mathop{\arg\min}_{\bm{c}}\frac{1}{2}\|\bm{W}\bm{u}-\bm{W}\bm{\Psi}\bm{c}\|_2^2 + \lambda\|\bm{c}\|_1,
\end{align} 
where $\lambda$ is a regularization parameter. 

The solution to these problems are closely related to the solution of either
\begin{align*}
\mathop{\arg\min}_{\bm{c}}\|\bm{c}\|_0 \mbox{ subject to } \|\bm{W}\bm{u}-\bm{W}\bm{\Psi}\bm{c}\|_2\le\delta,
\end{align*}
which is similar to (\ref{eqn:constrained}), or the closely related
\begin{align*}
\mathop{\arg\min}_{\bm{c}}\frac{1}{2}\|\bm{W}\bm{u}-\bm{W}\bm{\Psi}\bm{c}\|_2^2 + \lambda\|\bm{c}\|_0,
\end{align*}
which is similar to (\ref{eqn:regularized}). Here, $\|\bm{c}\|_0=\#(c_k\ne 0)$ is the number of non-zero entries of $\bm c$. Solutions to these problems are of great practical interest for sparse approximation and have received significant study in the field of Compressive Sampling/Compressed Sensing, see, e.g.,~\cite{Candes06a,Donoho06b,Elad10a,Eldar12a}, and more recently in UQ, \cite{Doostan10b,Doostan11a,Blatman11,Mathelin12a,Yan12,Yang13,Karagiannis14,Peng14,Schiavazzi14,Sargsyan14,Jones14a}.

\subsection{Contributions of This Work}
\label{subsec:Contribution}

This work is concerned with convergence analysis and sampling strategies to recover a sparse stochastic function in both Hermite and Legendre PC expansions from $\ell_1$-minimization problem (\ref{eqn:constrained}). As an extension of our previous work in \cite{Doostan10b,Doostan11a,Peng14}, the main contributions of this study are three-fold.

Firstly, we {\color{black}utilize properties} of these polynomials, in conjunction with the analysis of sparse function recovery in~\cite{CandesPlan,RauhutWard}, to give a framework which admits a bound on the number of samples sufficient for a successful solution of (\ref{eqn:constrained}). To our best knowledge, the Hermite results are the first of their type, and the Legendre recovery bounds, while here obtained from different techniques, are similar to those in~\cite{RauhutWard}.

Secondly, we provide a contribution of particular practical interest in that we analyze sampling Hermite polynomials uniformly over a $d$-dimensional ball -- with a radius depending on the order of approximation -- instead of sampling from the standard Gaussian measure. {\color{black}This sampling arises in a similar context to the Chebyshev distribution as a sampling for Legendre polynomials.} Interestingly, as explained in Section~\ref{subsubsec:asymmethod}, this sampling of Hermite polynomial expansion is analogous to {\it Hermite function} expansion,~\cite{Szego}, of appropriately weighted solution of interest. We provide analytic and numeric results justifying the use of this {\it importance sampling} distribution for the recovery of sparse Hermite PC expansions.

Finally, we analytically identify an importance sampling distribution with a statistical {\it optimality}, in terms of the {\it coherence} of the PC basis as a key recovery parameter of the method, and identify a Markov Chain Monte Carlo sampler for which we provide associated numeric results. This approach, here referred to as {\it coherence-optimal} sampling, provides a general sampling scheme for the reconstruction of sparse Hermite and Legendre PC expansions, and may be extended to other types of orthogonal bases. 

The motivation to design a sampling strategy based on the coherence is similar to that of~\cite{RauhutWard,Krahmer13}, but utilizing a different pre-conditioning from~\cite{Krahmer13}, considering unbounded bases and asymptotic scenarios, and providing a procedure for generating samples.

The presentation in this work has Section~\ref{sec:ProblemAndSolution} clearly stating the problem. Section~\ref{sec:motivation} provides key background information and motivates our approach, while Section~\ref{sec:sampling} describes our sampling methods and provides key theoretical results. Section~\ref{sec:examples} demonstrates the performance of the sampling methods and Section~\ref{sec:Proofs} presents the proofs to the Theorems from Section~\ref{sec:sampling}.

\section{Problem Statement and Solution Approach}
\label{sec:ProblemAndSolution}

We first describe the random inputs to the system, letting the random vector $\bm{\Xi}$, defined on the probability space $(\Omega,\mathcal{F},\mathbb{P})$, represent the input uncertainties to the physical problem under consideration. We assume that $(\Omega,\mathcal{F},\mathbb{P})$ is formed by the product of $d$ probability spaces $(\mathbb{R},\mathbb{B}(\mathbb{R}),\mathbb{P}_i)$ associated with each $\Xi_i$ where $\mathbb{B}$ denotes the Borel $\sigma$-algebra.  We note that this implies that $\mathcal{F}=\mathbb{B}(\mathbb{R}^d)$ the $d$-dimensional $\sigma$-algebra, and $\Omega=\mathbb{R}^d$. Further implied are that $\mathbb{P}$ is Lebesgue measurable and the $\Xi_i$ are independent random variables. For convenience, we assume that the $\Xi_i$ are identically distributed with distribution function $f(\xi)$, and abuse this notation by allowing that $\bm{\Xi}$ is distributed according to $f(\bm{\xi})$, noting that the two distributions may be differentiated by the presence of a scalar or vector function argument.

We consider the physical system through which the input uncertainty $\bm{\Xi}$ propagates to be given by operators defined on a bounded Lipschitz continuous domain $\mathcal{D}\subset\mathbb{R}^D$ for $D\in\{1,2,3\}$, with a boundary denoted by $\partial\mathcal{D}$. Letting operators $\mathcal{L},\mathcal{B}$ and $\mathcal{I}$ depend on the physics of the problem being considered, we assume that a solution $u$ satisfies
\begin{align*}
\mathcal{L}(\bm{x},t,\bm{\Xi};u(t,\bm{x},\bm{\Xi}))=0 &\qquad\bm{x}\in\mathcal{D},\\
\mathcal{B}(\bm{x},t,\bm{\Xi};u(t,\bm{x},\bm{\Xi}))=0 &\qquad\bm{x}\in\partial\mathcal{D},\\
\mathcal{I}(\bm{x},0,\bm{\Xi};u(0,\bm{x},\bm{\Xi}))=0 &\qquad\bm{x}\in\mathcal{D}.
\end{align*}
We note that the problems considered in Section~\ref{sec:examples} depend only on space or time, but the methods considered here are independent of the underlying physical problem.
We assume that conditioned on the $i$th independent sample of $\bm{\Xi}$, denoted by $\bm{\xi}^{(i)}$, a numerical solution to this problem may be identified by a fixed solver; we utilize FEniCS~\cite{FEniCS} for the examples in the present work. For any fixed $\bm{x}_0\in\mathcal{D}$ and $t_0> 0$, our objective is to reconstruct $u(\bm{x}_0,t_0,\bm{\Xi})$ via problem (\ref{eqn:constrained}) using information obtained from $\{u(\bm{x}_0,t_0,\bm{\xi}^{(i)})\}_{i=1}^N$, that is $N$ independent realizations of the QoI. For a cleaner presentation, we suppress the dependence of $u$ on $\bm{x}_0$ and $t_0$. 

\subsection{Approximately Sparse PCE}
\label{subsec:PCE}

Here we discuss the polynomials in (\ref{Eq:PCETrunc}) as utilized in this work, and the key sparsity assumption which this approximation frequently facilitates. We consider an arbitrary number of input dimensions, denoted by $d$, and the set of orthogonal polynomials in any mixture of these coordinates of total order less than or equal to $p$. To explain the total order, let {\color{black}$\bm{k} = (k_1,\dots,k_d)$} be a $d\times 1$ multi-index such that $k_i\in{\color{black}\mathbb{N}\cup\{0\}}$ represents the order of the polynomial $\psi_{k_i}(\Xi_i)$, orthogonal with respect to the measure of $\Xi_i$. The $d$-dimensional polynomials $\psi_{\bm{k}}(\bm{\Xi})$ are constructed by the tensorization of {\color{black}$\psi_{k_i}(\Xi_i)$},
\begin{align*}
\psi_{\bm{k}}(\bm{\Xi})=\mathop{\prod}\limits_{i=1}^d\psi_{k_i}(\Xi_i).
\end{align*}
The total order of $p$ implies that we consider all polynomials satisfying
\begin{align*}
\|\bm{k}\|_1\le p\qquad k_i\in{\color{black}\mathbb{N}\cup\{0\}}\quad\forall i.
\end{align*}
We note that a direct combinatorial count implies that a $d$-dimensional approximation of total order $p$ has $P={p+d\choose d}$ basis polynomials. This total order basis facilitates a polynomial approximation to the general function that favors lower order polynomials. If the coefficients have a sufficiently rapid decay or if certain dimensions are dominant in an accurate reconstruction, then we have
\begin{align*}
u(\bm{\Xi})\approx \mathop{\sum}\limits_{\bm{k}\in\mathcal{C}} c_{\bm{k}} \psi_{\bm{k}}(\bm{\Xi}),
\end{align*}
where $s:=|\mathcal{C}| \ll P$ is the operative sparsity of the approximation that leads to stable and convergent approximations of $\bm{c}$ from (\ref{eqn:constrained}), using $N<P$ random samples of $u(\bm\Xi)$. To demonstrate this, we rely primarily on existing theorems from~\cite{CandesPlan} as presented in Section~\ref{sec:motivation}. Subsequently, in Section \ref{sec:sampling}, we adapt these results to the case of sparse PC expansions for which we utilize basic properties of orthogonal polynomials, and further present our main results on the choices of random sampling of $\bm\Xi$ and, hence, $u(\bm\Xi)$.\\
 
\noindent{\bf Notation:}  In the sequel we occasionally use a multi-index notation for polynomials, but also find it convenience to index polynomials by a scalar, e.g., $k$, from $1$ to $P$.

\section{\texorpdfstring{Definitions and Background}{Definitions and Background}}
\label{sec:motivation}

To contextualize the results from~\cite{CandesPlan}, presented in Section~\ref{sec:candes_theorems}, we first introduce two main definitions that are used both in these results, as well as in constructing our sampling methods.

\subsection{\texorpdfstring{Sampling Definitions}{Sampling Definitions}}
\label{subsubsec:sampling}
We first consider the set of polynomials, $\{\psi_k(\bm{\xi})\}_{k=1}^{P}$, as defined in Section~\ref{sec:ProblemAndSolution} and define $B(\bm{\xi})$ to be
\begin{align}
\label{eqn:btspec}
B(\bm{\xi}):=\mathop{\max}\limits_{k=1:P}|\psi_k(\bm{\xi})|.
\end{align}
This represents a uniformly least upper bound on the basis polynomials of interest. In addition, we consider 
\begin{align}
\label{eqn:GDef}
G(\bm{\xi})\ge B(\bm{\xi}) \qquad \forall\bm{\xi}\in\Omega.
\end{align}
Here, $G(\bm{\xi})$ represents an upper bound on the tight bound of $B(\bm{\xi})$ for all $\bm{\xi}\in\Omega$, where $\Omega$ is the sample space of potential values for $\bm{\xi}$ as a realization of $\bm{\Xi}$.

We note that for several orthonormal polynomials of interest a bound on $B(\bm{\xi})$ may be attained,~\cite{RauhutWard,Szego,Hermite1,AskeyWainger,LagAsym,Jacobi}. 
In this case, we have that $\psi_k(\bm{\xi})/G(\bm{\xi})\le 1$. It follows that for any set $\mathcal{S}\subseteq\Omega$,
\begin{align}
\label{eqn:normalizingconstant}
c = \left(\int_{\mathcal{S}}f(\bm{\xi})G^2(\bm{\xi})d\bm{\xi}\right)^{-1/2}
\end{align}
is such that
\begin{align*}
c^2\int_{\mathcal{S}}f(\bm{\xi})G^2(\bm{\xi})d\bm{\xi}=1,
\end{align*}
and 
\begin{align}
\label{eqn:optimal_pdf_gen}
f_{\bm{Y}}(\bm{\xi}):=c^2f(\bm{\xi})G^2(\bm{\xi}),
\end{align}
is a probability distribution supported on $\mathcal{S}$, which we consider as the distribution for $\bm{Y}$. Let $\delta_{i,j}$ denote the Kronecker delta such that $\delta_{i,j}=1$ if $i=j$ and $0$ if $i\ne j$. Note that for $i,j=1:P$,
\begin{align}
\label{Eqn:TransformIntegral}
\left|\int_{\mathcal{S}}\frac{\psi_i(\bm{\xi})}{cG(\bm{\xi})}\frac{\psi_j(\bm{\xi})}{cG(\bm{\xi})}c^2f(\bm{\xi})G^2(\bm{\xi})d\bm{\xi}-\delta_{i,j}\right|&\le \epsilon_{i,j},
\end{align}
and we may select $\mathcal{S}$ such that $\epsilon_{i,j}$ may be made as small as needed, e.g., if we take $\mathcal{S} = \Omega$, then $\epsilon_{i,j}=0$. For this purpose we employ the heuristic of selecting $\mathcal{S}$ to encompass the largest values of $f(\bm{\xi})$ until $\mathcal{S}$ is large enough to satisfy the condition (\ref{Eqn:CoherenceUnbounded}), discussed in Section~\ref{subsubsec:Coherence}. The justification for this is that in unbounded domains, e.g., for Hermite polynomials, regions of small $f(\bm{\xi})$ typically correspond to larger $\sup_{k=1:P}|\psi_k(\bm{\xi})|$ as $p$ grows~\cite{Szego,AskeyWainger,Hermite1}.

While this formulation is useful for identifying distributions for $\bm{Y}$, unfortunately, we may no longer guarantee that $\mathbb{E}[\psi_i(\bm{Y})\psi_j(\bm{Y})]-\delta_{i,j}$ is small. Fortunately, from (\ref{Eqn:TransformIntegral})  if we let
\begin{align}
\label{eqn:weightFunction}
w(\bm{Y})&:=\frac{1}{cG(\bm{Y})},
\end{align}
then $|\mathbb{E}[w^2(\bm Y)\psi_i(\bm{Y})\psi_j(\bm{Y})]-\delta_{i,j}|\le \epsilon_{i,j}$. In this way we consider $w(\bm{Y})$ to be a weight function so that $\{w(\bm{Y})\psi_i(\bm{Y})\}_{i=1}^P$, are approximately orthonormal random variables. This function defines the diagonal positive-definite matrix $\bm{W}$ from (\ref{eqn:constrained}) as
\begin{align*}
 \bm{W}(i,i)= w(\bm{\xi}^{(i)}),
\end{align*}
where $\bm{\xi}^{(i)}$ is the $i$th realization of $\bm{Y}$. For a notational symmetry with the conceptual connection, we refer to all realized random vectors by $\bm{\xi}$ regardless of the sampling distribution for $\bm{\xi}$, noting that the weight function, $w$, depends on that distribution. Additionally, we note that for simulation, we are not interested in the normalizing constant, $c$, associated with describing our sampling distribution.

\subsection{\texorpdfstring{Coherence Definition}{Coherence Definition}}
\label{subsubsec:Coherence}
Consider realizations of $w(\bm{Y})\psi_k(\bm{Y})$ for $k=1:P$. We investigate the coherence parameter defined as in~\cite{CandesPlan} by
\begin{align}
\label{Eqn:CoherenceBounded}
\mu(\bm{Y}) &:= \sup_{k=1:P,\bm{\xi}\in\Omega}|w(\bm{\xi})\psi_k(\bm{\xi})|^2.
\end{align}
This is a conceptually simple parameter that we will see allows us to bound the number of samples necessary to accurately recover $\bm{c}$ via a solution to (\ref{eqn:constrained}). From (\ref{eqn:weightFunction}) {\color{black}and (\ref{Eqn:CoherenceBounded})} we are motivated to take $G(\bm{\xi})$ to be $B(\bm{\xi})$ as defined in (\ref{eqn:btspec}), and as we shall show in Section \ref{subsec:MCMC}, this choice leads to an optimally minimal coherence. Fortunately, asymptotic results give us approximations to the distribution $f_{\bm Y}(\bm\xi)$ of $\bm{Y}$ in certain cases. These approximations also lead to easier simulation of $\bm{Y}$ {\color{black}when $f_{\bm Y}(\bm\xi)$ corresponds to the choice of $G(\bm{\xi})=B(\bm{\xi})$, as described in Section \ref{subsubsec:MCMCmethod}.} 

We utilize the definition in (\ref{Eqn:CoherenceBounded}) when analyzing Legendre polynomials which are bounded on the domain $[-1,1]^d$. However, we note that (\ref{Eqn:CoherenceBounded}) is not useful when $\sup_{k=1:P,\bm{\xi}\in\Omega}|w(\bm{\xi})\psi_k(\bm{\xi})|^2$ is infinite, such as when $\psi_k(\bm{\xi})$ are Hermite polynomials and $w(\bm{\xi}) = 1$. If $N$ is the number of samples of $\bm{Y}$ which we will take, following \cite{CandesPlan}, we consider a truncation of $\Omega$ to some appropriate $\mathcal{S}$ and let
\begin{align}
\label{Eqn:CoherenceUnbounded}
\mu(\bm{Y}) &:= \mathop{\min}\limits_{\mathcal{S}}\left\{\sup_{k=1:P,\bm{\xi}\in\mathcal{S}}|w(\bm{\xi})\psi_k(\bm{\xi})|^2,\cdots\right.\\
\nonumber
\mbox{subject to }&\left.\mathbb{P}(\mathcal{S}^c)<\frac{1}{NP};\ \mathop{\sum}\limits_{k=1}^{P}\mathbb{E}\left[|w(\bm{Y})\psi_k(\bm{Y})|^2\bm{1}_{\mathcal{S}^{c}}\right]\le\frac{1}{20}P^{-1/2}\right\},
\end{align}
where $\mathcal{S}$ is a subset of the support of $f$, a superscript $c$ denotes a set complement, and $\bm{1}$ is the indicator function. While (\ref{Eqn:CoherenceBounded}) highlights the quantity that we seek to bound, the conditions in (\ref{Eqn:CoherenceUnbounded}) insure that the truncation from $\Omega$ to $\mathcal{S}$ has a limited effect on the orthogonality of the set of random variables, $\{w(\bm{Y})\psi_k(\bm{Y})\}_{k=1}^P$. As normal random variables are unbounded, we use (\ref{Eqn:CoherenceUnbounded}) in the analysis of Hermite polynomials with a truncation $\mathcal{S}$ that captures the essential behavior of $w(\bm{\xi})\psi_k(\bm{\xi})$. These definitions are compatible in that either definition may be used for the following theorems.

\subsection{\texorpdfstring{Convergence Theorems}{Convergence Theorems}}
\label{sec:candes_theorems}

The following theorems use the coherence parameter in either (\ref{Eqn:CoherenceBounded}) or (\ref{Eqn:CoherenceUnbounded}) to bound the number of samples necessary to recover a sparse signal with high probability.
\label{subsec:coherenceTheorems}
\begin{thm}
\label{thm:SampleDepth}\cite{CandesPlan}
Let $\bm{c}$ be a fixed arbitrary vector in $\mathbb{R}^{P}$ with at most $s$ non-zero elements such that $\bm{\Psi c} = \bm{u}$, where $\bm\Psi$ is defined as in (\ref{eqn:psi_u}). With probability at least $1-5/P-e^{-\beta}$, and $C$ an absolute constant, if
\begin{align}
\label{eqn:NecSamp}
N\ge C(1+\beta)\mu(\bm{Y})s\log(P),
\end{align}
then $\bm{c}=\arg\min_{\bm{c}}\{\|\bm{c}\|_1 :\bm{W\Psi c} = \bm{Wu}\}$.
\end{thm}

When allowing for truncation error and considering a regularized version of this $\ell_1$-minimization problem as in (\ref{eqn:regularized}), a similar result may be stated. Following \cite{CandesPlan}, we require the condition that $\|\bm{\Psi}^{T}\bm{W}^2\bm{z}\|_{\infty}\le\nu$, for some $0\le \nu<\infty$, where $\bm{z}$ is the associated truncation error with the model $\bm{u}=\bm{\Psi}\bm{c}+\bm{z}$ for an arbitrary solution vector $\bm{c}$. Additionally, we denote by $\sigma_w$ the standard deviation of the weighted truncation error $w(\bm Y)z(\bm Y)$.

\begin{thm}
\label{thm:SampleDepthNoise}\cite{CandesPlan}
Let $\bm{c}$ be a fixed arbitrary vector in $\mathbb{R}^{P}$, and $\bm{c}_s$ be a vector such that $\bm{c}_s(i) = \bm{c}(i)$ for the $s$ largest $|\bm{c}(i)|$, and $\bm{c}_s(i) = 0$ otherwise. For some $\bar{s}$, let 
\begin{align*}
N\ge C(1+\beta)\mu(\bm{Y})\bar{s}\log(P).
\end{align*}
With probability at least $1-6/P-6e^{-\beta}$, and $C$ an absolute constant, the solution to 
\begin{align*}
\hat{\bm{c}}=\mathop{\min}\limits_{\bar{\bm{c}}\in\mathbb{R}^P} \frac{1}{2}\|\bm{W\Psi}\bar{\bm{c}}-\bm{Wu}\|_2^2+\lambda\sigma_w\|\bar{\bm{c}}\|_1,
\end{align*}
with $\lambda = 10\sqrt{\frac{\log(P)}{N}}$ obeys for any $\bm{c}$,
\begin{align*}
\|\hat{\bm{c}}-\bm{c}\|_2 &\le \mathop{\min}\limits_{1\le s\le \bar{s}}C(1+\alpha)\left[\frac{\|\bm{c}-\bm{c}_s\|_1}{\sqrt{s}}+\sigma_w\sqrt{\frac{s\log(P)}{N}}\right];\\
\|\hat{\bm{c}}-\bm{c}\|_1 &\le \mathop{\min}\limits_{1\le s\le \bar{s}}C(1+\alpha)\left[\|\bm{c}-\bm{c}_s\|_1+\sigma_w s\sqrt{\frac{\log(P)}{N}}\right],
\end{align*}
where $\alpha:=\sqrt{\frac{(1+\beta)s\log^5(P)}{N}}$.
\end{thm}

We note that when $\|\bm{\Psi}^{T}\bm{W}^2\bm{z}\|_{\infty}$ cannot be bounded by a $\nu$, we may be interested in a subset $\mathcal{S}$ of $\Omega$ that will be sampled with sufficiently high probability and admit a bound on $\|\bm{\Psi}^{T}\bm{W}^2\bm{z}\bm{1}_{(\bm{Y}_1,\cdots,\bm{Y}_N)\in\mathcal{S}}\|_{\infty}$. This may be related to the truncation of $\Omega$ to $\mathcal{S}$ in the conditions of (\ref{Eqn:CoherenceUnbounded}).

These results show how a bound on $\mu(\bm{Y})$ translates into a bound on the number of samples needed to recover a solution vector, and provide a theoretical justification to the identification of distributions for $\bm{Y}$ which yield a smaller bound on $\mu(\bm{Y})$. With these bounds we may utilize Theorems~\ref{thm:SampleDepth} and~\ref{thm:SampleDepthNoise} to bound the number of samples required to recover solutions of any particular sparsity.
\section{Sampling Methods}
\label{sec:sampling}

Here we describe the sampling methods that we consider in this work, and present theorems related to recovery when we use them. We first consider a sampling according to random variables defined by the orthogonality measure in Section~\ref{subsec:std}. Such a sampling, dubbed here  {\it standard sampling}, is commonly used in PC regression, \cite{Hosder06,LeMaitre10,Doostan11a,Mathelin12a}. Second, we consider sampling from a distribution {\color{black}related to} an asymptotic analysis of the orthogonal polynomials $\psi_k(\bm\Xi)$ in Section~\ref{subsec:asym}, and refer to it as {\it asymptotic sampling}. Finally, in Section~\ref{subsec:MCMC}, we introduce the {\it coherence-optimal sampling} that corresponds to minimizing the coherence parameters defined in Section~\ref{subsubsec:Coherence}.

\subsection{Standard Sampling}
\label{subsec:std}

Here we consider sampling $\bm{\xi}$ according to $f(\bm\xi)$, the distribution with respect to which the PC bases are naturally orthogonal. This implies taking $w(\bm\xi) = 1$.%

\subsubsection{Standard Sampling Method}
\label{subsubsec:stdmethod}
For the $d$-dimensional Legendre polynomials the standard method corresponds to sampling from the uniform distribution on $[-1,1]^d$, while for $d$-dimensional Hermite polynomials this corresponds to samples from a multi-variate normal distribution such that each of $d$ coordinates is an independent standard normal random variable.

\subsubsection{Theorems}
\label{subsubsec:stdTheorems}

A standard sampling of Hermite polynomials leads to a coherence bounded as in Theorem~\ref{thm:NatSampleCoherence}, while a standard sampling of Legendre polynomials leads to a coherence bounded as in Theorem~\ref{thm:NatSampleCoherenceLeg}.  We note that these results hold for a number of dimensions $d$ and a set of orthogonal polynomials of arbitrary total order $p$ as defined in Section~\ref{subsec:PCE}. 

\begin{thm}
\label{thm:NatSampleCoherence}
Assume that $d=o(p)$, that is, $d$ is asymptotically dominated by $p$. Additionally, let $N = O(P^k)$ for some $k>0$, that is, the number of samples does not grow faster than a polynomial in the number of basis polynomials considered. For $d$-dimensional Hermite polynomials of total order $p\ge 1$, the coherence in (\ref{Eqn:CoherenceUnbounded}) is bounded by 
\begin{align}
\mu(\bm{\Xi})&\le C_p\cdot \eta_p^p,
\end{align}
for some constants $C_p,\eta_p$ depending on $p$. For $d = o(p)$, and as $p\rightarrow\infty$, we may take $C_p$ and $\eta_p$ to be larger than but arbitrarily close to $1$ and $\exp(2-\log(2))\approx 3.6945$, respectively.
\end{thm}

We note that together with Theorems~\ref{thm:SampleDepth} and~\ref{thm:SampleDepthNoise}, this implies that with high probability, the number of samples required for recovery from Hermite polynomials grows exponentially with the total order of approximation.  The following theorem for Legendre polynomials is analogous to previous results in~\cite{RauhutWard}, and provides a similar result for the number of samples required for signal recovery.\\

\noindent{\bf Remark:} When sampling Hermite polynomials, we have the technical requirement that $N=O(P^k)$ for some finite $k$, and we note that this condition is satisfied here as $N<P$ is the case of interest in compressive sampling.

\begin{thm}
\label{thm:NatSampleCoherenceLeg}
A standard sampling of the $d$-dimensional Legendre polynomials of total order $p$ gives a coherence of
\begin{align}
\label{eqn:LegNatBound}
\mu(\bm{\Xi})&\le \exp(2p).
\end{align}
\end{thm}

As we shall see in Section \ref{subsec:JacProofNatSamp}, for the case of $p<d$ the bound in (\ref{eqn:LegNatBound}) may be improved to $\mu(\bm{\Xi})\le 3^p \approx \exp(1.1p)$. Additionally, for $p>d$, the bound in (\ref{eqn:LegNatBound}) is loose, but a sharper dimension-dependent bound is given by $(2p/d+1)^d$.

\subsection{Asymptotic Sampling}
\label{subsec:asym}

{\color{black}Here we consider taking $G(\bm{\xi})$ to approximate or coincide with the asymptotic (in order) envelope for the polynomials as the order $p$ goes to infinity. Specifically, for the case of Hermite polynomials we consider a relatively simple envelope function over a significant range of $\bm\xi$, corresponding to a uniform sampling, though this envelope does not coincide with $B(\bm\xi)$ and is loose compared to known behavior of Hermite polynomials at high orders, \cite{krasikov2004new}. The uniform approximation is, however, both simple to simulate and analyze. For the case of Legendre polynomials, we take $G(\bm{\xi})$ to be $B(\bm\xi)$ for asymptotically large order $p$, which corresponds to Chebyshev sampling. For both cases, sampling with this choice of $G(\bm\xi)$ leads to coherence parameters with weaker dependence on $p$, as compared to the standard sampling.}


%
\subsubsection{Asymptotic Sampling Method}
\label{subsubsec:asymmethod}

For $d$-dimensional Hermite polynomials, {\color{black}we sample uniformly from within the $d$-dimensional ball of radius $\sqrt{2}\sqrt{2p+1}$, which corresponds to $G(\bm\xi)=\exp(\|\bm{\xi}\|_2^2/4)$ on this ball.} This choice of {\color{black}uniform sampling and} radius is motivated by the analysis of Section~\ref{subsec:keyHermite}. For completion, we outline one algorithm for sampling uniformly from the $d$-dimensional ball of radius $r$. First, let $\bm{Z}:=(Z_1,\cdots,Z_d)$ be a vector of $d$ independent normally distributed random variables with zero mean and the same variance. If $U$ is another independent random variable that is uniformly distributed on $[0,1]$, then
\begin{align*}
\bm{Y}&:=\frac{\bm{Z}}{\|\bm{Z}\|_2}rU^{1/d},
\end{align*}
represents a random sample uniformly distributed within the $d$-dimensional ball of radius $r$. This may be verified as $\bm{Z}/\|\bm{Z}\|_2$ is uniformly distributed on the $d$-dimensional hypersphere, while $rU^{1/d}$ is the distribution for the radius of the realization within the ball that coincides with a uniform sampling within the ball. Additionally, this leads to a weight function given by
\begin{align*}
w(\bm{\xi}):= \exp(-\|\bm{\xi}\|_2^2/4).
\end{align*}

\noindent{\bf Remark} ({\it Connection with Hermite function expansion}). We highlight that the application of the weight function $w(\bm{\xi})= \exp(-\|\bm{\xi}\|_2^2/4)$ to the Hermite polynomials $\psi_k(\bm \xi)$ leads to the so called {\it Hermite functions}, i.e., $\exp(-\|\bm{\xi}\|_2^2/4)\psi_k(\bm \xi)$, that are orthogonal with respect to the uniform measure,~\cite{Szego}. This implies that the Hermite polynomial expansion with asymptotic sampling is analogous to Hermite function expansion of $w(\bm{\Xi})u(\bm\Xi)$. Notice that in a standard Hermite function expansion, the solution of interest, $u(\bm{\Xi})$, is expanded in $\{\exp(-\|\bm{\Xi}\|_2^2/4)\psi_k(\bm \Xi)\}$. The only computational difference between solving for a Hermite polynomial expansion under this sampling and a Hermite function expansion, is whether, during computation of the coefficients, the realized $u(\bm{\xi})$ are multiplied by $w(\bm{\xi})$ or not. \\[-.3cm]

For the $d$-dimensional Legendre polynomials this corresponds to sampling from the Chebyshev distribution on $[-1,1]^d$,~\cite{RauhutWard}, that is the distribution in each of $d$ coordinates is
\begin{align*}
f_Y(\xi)&:=\frac{1}{\pi\sqrt{1-\xi^2}},
\end{align*}
for $\xi\in[-1,1]$. Each coordinate is easily simulated from $\cos(\pi U)$ where $U$ is uniformly distributed on $[0,1]$. Additionally, this leads to a weight function given by
\begin{align*}
 w(\bm{\xi}):=\mathop{\prod}\limits_{i=1}^d(1-\xi_i^2)^{1/4}.
\end{align*}
\subsubsection{Theorems}
\label{subsubsec:asymTheorems}

{\color{black}Analysis of the Hermite and Legendre polynomials sampled according to these alternative distributions leads to a coherence with a weaker asymptotic dependence on $p$. In Theorem~\ref{thm:TransformedSamples} and Theorem~\ref{thm:TransformedSamplesLeg} we quantify such a dependence.}

\begin{thm}
\label{thm:TransformedSamples}
Assume that $N = O(P^k)$ for some $k>0$, that is the number of samples does not grow faster than a polynomial in the number of basis polynomials considered. We note that this includes the important and common case that $N\le P$. Let $V(r,d)=(r\sqrt{\pi})^d/\Gamma(d/2+1)$ denote the volume inside the hypersphere with radius $r$ in dimension $d$. 

For the sampling of Hermite polynomials, sampling uniformly from the $d$-dimensional ball of radius $\sqrt{2}\sqrt{(2+\epsilon_{p})p+1}$, and weighting realized $\psi_k(\bm{\xi}^{(i)})$ on this ball by $w(\bm\xi^{(i)}) = \exp(-\|\bm{\xi}^{(i)}\|_2^2/4)$, gives
\begin{equation}
\mu(\bm{Y})=O(\pi^{-d/2}V(\sqrt{2p},d))=O((2p)^{d/2}/\Gamma(d/2+1)). 
\end{equation}
Here, we note that $\epsilon_{p}\rightarrow 0$ if $d = o(p)$, and that the radius of the sampling is a factor of $\sqrt{2}$ times larger than the radius of the volume in the coherence, due to a normalization explained in Section~\ref{sec:Proofs}.
\end{thm}

In the uniform sampling in this work we set $\epsilon_{p}$ in Theorem~\ref{thm:TransformedSamples} to be zero, leaving as an open problem the determination of an optimal $\epsilon_{p}$, and hence sampling radius for uniform sampling of Hermite polynomials. Additionally, this theorem is applicable to sampling Hermite functions with a standard, i.e. uniform, sampling as $w(\bm{\xi})\psi_k(\bm{\xi})$ is a Hermite function.

In the case of Legendre polynomials sampled by Chebyshev distribution we have a complete independence of the order of approximation, which agrees with previous results in~\cite{RauhutWard}.

\begin{thm}
\label{thm:TransformedSamplesLeg}
For the sampling of $d$-dimensional Legendre polynomials according to the $d$-dimensional Chebyshev distribution and weight $\psi_k(\bm{\xi})$ proportional to $w(\bm\xi) = \prod_{i=1}^d(1-\xi_i^2)^{1/4}$, regardless of the relationship between $d$ and $p$, we have that
\begin{align}
\label{eqn:LegTransformedSample}
\mu(\bm{Y})&\le 3^d.
\end{align}
\end{thm}

It is worthwhile highlighting that the combination of Theorems \ref{thm:NatSampleCoherenceLeg} and \ref{thm:TransformedSamplesLeg} suggests sampling Legendre polynomials by uniform distribution when $d>p$ and Chebyshev distribution when $d<p$. A similar observation has been made in \cite{Yan12}.

\subsection{Coherence-optimal Sampling}
\label{subsec:MCMC}

Here we consider taking $G(\bm{\xi})=B(\bm{\xi})$ in (\ref{eqn:optimal_pdf_gen}), which implies sampling $\bm{\xi}$ according to the distribution
\begin{equation}
\label{eqn:optimal_pdf}
f_{\bm Y}(\bm \xi)=c^2f(\bm{\xi})B^2(\bm{\xi}),
\end{equation}
with some appropriate normalizing constant $c$. Corresponding to this sampling, we apply the weight function
\begin{align*}
w(\bm{\xi})=\frac{1}{B(\bm{\xi})}.
\end{align*}
Notice that in (\ref{eqn:optimal_pdf}), $f(\bm{\xi})$ is the measure with respect to which the polynomials $\psi_k(\bm\xi)$ are naturally orthogonal. 

\subsubsection{MCMC Sampling Method}
\label{subsubsec:MCMCmethod}

While $B(\bm{\xi})$, as defined in (\ref{eqn:btspec}), is straightforward to evaluate for a fixed $\bm{\xi}$ by iterating over each $k=1:P$, the quantity is difficult to evaluate over a range of $\bm{\xi}$, thus making it difficult to accurately compute the normalizing constant $c$ in (\ref{eqn:optimal_pdf}). This motivates sampling $\bm{\Xi}$ from (\ref{eqn:optimal_pdf}) via a Monte Carlo Markov Chain (MCMC) approach, specifically using the Metropolis-Hastings sampler,~\cite{Hastings}. The MCMC method uses the computable point-wise evaluation of $B(\bm{\xi})$, and does not require an identification of $c$ necessary to normalize to a probability distribution. Additionally, this sampling distribution allows the easy evaluation of $w(\bm{\xi})$ using only the realized sample.

The MCMC sampler requires a proposal, or candidate, distribution and when $p>d$ we suggest those obtained from Section~\ref{subsec:asym}, giving a uniform sampling on a $d$-dimensional ball for Hermite polynomials, and $d$-dimensional Chebyshev sampling for Legendre polynomials. Similarly, when $p\le d$ we suggest those obtained from Section~\ref{subsec:std}, giving a standard normal sampling for Hermite polynomials, and sampling uniformly for $[-1,1]^d$ for Legendre polynomials. We follow these proposal distributions for the sampling which we do in this work. Note that each proposal distribution covers the entire domain $\mathcal{S}$, and if the proposal and target distribution approximately match, then the acceptance rate is high and few burn-in samples are needed to approximately draw from the desired distribution for $\bm{Y}$. There is interest in identifying better proposal distributions, to be studied further. One caveat which we note is that the proofs of Theorems~\ref{thm:SampleDepth} and~\ref{thm:SampleDepthNoise} require independent sampling, so that it is proper to restart a chain after each 
accepted sample, but a more practical method is to discard intermediate samples so that serial dependence is small,~\cite{MCMCBook}. We note that in applications where evaluation of the QoI is expensive, the generation of the samples, $\{\bm{\xi}^{(i)}\}_{i=1}^N$, is not typically a bottleneck, so that the extra cost of MCMC sampling is frequently acceptable in practice.

\subsubsection{Theorem}
\label{subsubsec:MCMCTheorem}

Theorem~\ref{thm:LowPTransformedSamples} justifies the intuition that taking $G(\bm{\xi})$ associated with sampling to be the envelope function $B(\bm{\xi})$ leads to a minimal $\mu(\bm{Y})$.

\begin{thm}
\label{thm:LowPTransformedSamples} 
Let $\mathcal{S}$ be a set chosen to satisfy the conditions of (\ref{Eqn:CoherenceUnbounded}) implying that no subset $\mathcal{S}_s$ of $\mathcal{S}$ with $\mu(\mathcal{S}_s)<\mu(\mathcal{S})$ satisfies the conditions of (\ref{Eqn:CoherenceUnbounded}).
Let $B(\bm{\xi})$ be as in (\ref{eqn:btspec}). If we sample from the distribution proportional to $f(\bm{\xi})B^2(\bm{\xi})$ and weight $\psi_k(\bm{\xi})$ proportional to $w(\bm\xi) =1/B(\bm{\xi})$, then the coherence parameter achieves a minimum over all sampling schemes of $\psi_{k}(\bm{\xi})$, $k=1:P$, and distributions supported on $\mathcal{S}$.
\end{thm}

In the next section we explore how the sample distributions associated with these results perform when used to approximate sparse functions in the appropriate PC basis.

\section{Numerical Examples}
\label{sec:examples}

Here we numerically investigate the different sampling schemes discussed in Section~\ref{sec:sampling}, considering the coherence parameter in Section \ref{subsec:ComputedCoherence}, randomly generated manufactured sparse functions in Section~\ref{subsec:RandomSignals}, the solution to an elliptic PDE with random coefficient in Section~\ref{subsec:Elliptic}, and the amount of reaction at a given time in an adsorption model from \cite{Makeev02} in Section~\ref{subsec:EllipticOperator}.
\subsection{Computed Coherence}
\label{subsec:ComputedCoherence}
The coherence parameter of Section~\ref{subsubsec:Coherence} can be estimated from a large sample of realized points. Doing so leads to the results in Figure~\ref{Fig:HermCoh} for Hermite polynomials, and those in Figure~\ref{Fig:LegCoh} for Legendre polynomials. We consider three sampling schemes, the standard scheme where we sample based on the underlying distribution of random variables in question as in Section~\ref{subsubsec:stdmethod}; an asymptotically motivated method to insure a coherence with weaker dependence on the order $p$ as in Section~\ref{subsubsec:asymmethod}; and a coherence-optimal sampling based on the distribution proportional to the envelope of basis functions as in Section~\ref{subsubsec:MCMCmethod}. We see that standard sampling tends to perform poorly at high orders, while asymptotic sampling tend to perform poorly for high-dimensional problems. Additionally, coherence-optimal sampling performs well in all regimes. These observations are consistent with the theoretical results 
presented in Section \ref{sec:sampling}.

\begin{figure}[h!]
\centering
\includegraphics[trim=15mm 10mm 10mm 10mm, clip, scale=0.35]{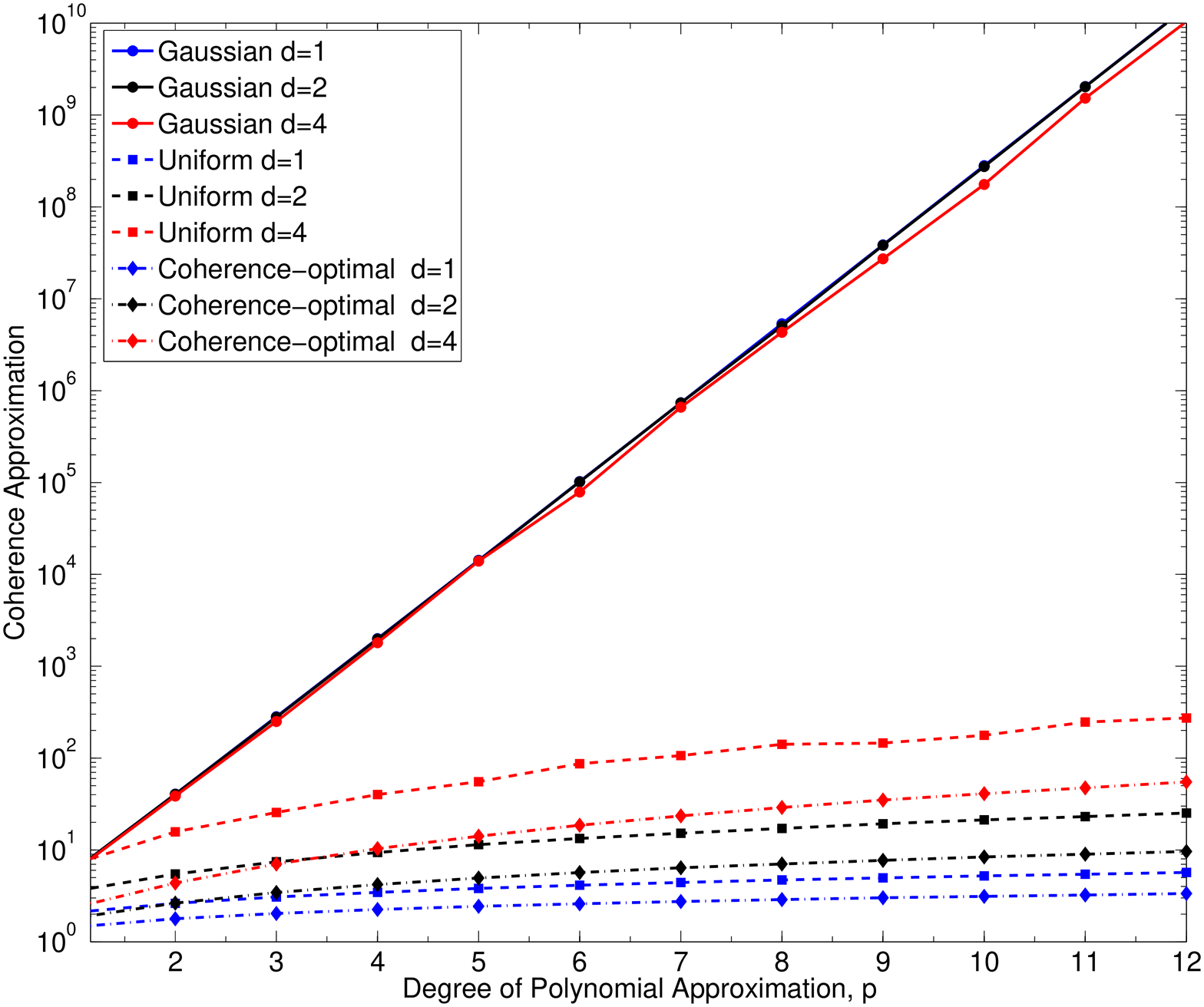}
\caption{Computed $\mu(\bm{Y})$ for different sampling methods of Hermite polynomials for different $d$ and $p$.}
\label{Fig:HermCoh}
\end{figure}
\begin{figure}[h!]
\centering
\includegraphics[trim=15mm 10mm 10mm 10mm, clip, scale=0.35]{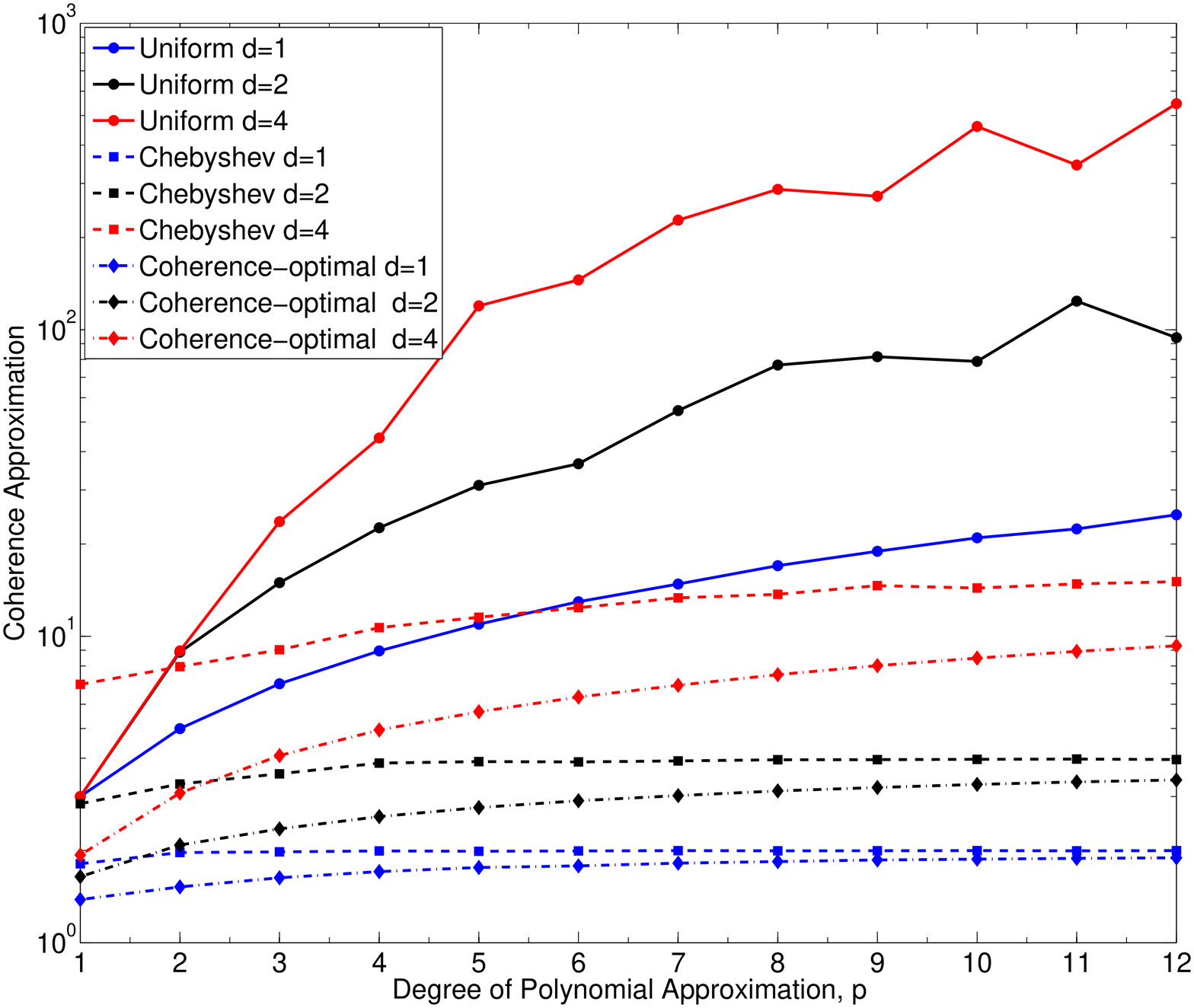}
\caption{Computed $\mu(\bm{Y})$ for different sampling methods of Legendre polynomials for different $d$ and $p$.}
\label{Fig:LegCoh}
\end{figure}
\subsection{Manufactured Sparse Functions}
\label{subsec:RandomSignals}

In this section, we investigate the reconstruction accuracy of the competing sampling schemes on randomly generated sparse solution vectors, $\bm{c}$, such that $\bm{\Psi}\bm{c}=\bm{u}$. Here, $\bm{c}$ is chosen to have a uniformly selected random support and independent standard normal random variables for {\color{black}values of} each supported coordinate. We measure reconstruction accuracy as a function of sparsity, denoted by $s$, and the number of independent samples of $\bm{Y}$, denoted by $N$. We declare $\hat{\bm{c}}$ to be a successful recovery of $\bm{c}$ if $\|\hat{\bm{c}}-\bm{c}\|_2/\|\bm{c}\|_2\le 0.01$, where $\hat{\bm{c}}$ is a solution to (\ref{eqn:constrained}) and in this work is computed using the $\ell_1$-minimization solver of SparseLab~\cite{SparseLab}, which is based on a primal-dual interior-point method. Each success probability is calculated from 2500 independent realizations of $\bm\Psi$ and $\bm c$.

For a more comparable presentation, we normalize the number of samples by the number of basis functions considered, $N/P\in[0.1,1]$, and similarly normalize the sparsity by the number of samples, $s/N\in[0.1,1]$. To compare the ability to recover solutions, we identify the probability of recovery on a $90\times 90$ uniform grid in $(N/P,s/N)$ for different $(d,p)$ pairs as well as for different distributions of $\bm{Y}$. The results are presented in Figures~\ref{Fig:HermRand} and~\ref{Fig:LegRand} for Hermite and Legendre polynomials, respectively, where we consider three sampling schemes of Sections~\ref{subsubsec:stdmethod},~\ref{subsubsec:asymmethod}, and~\ref{subsubsec:MCMCmethod}. For the coherence-optimal sampling, in conjunction with Theorem~\ref{thm:LowPTransformedSamples}, we use a Metropolis-Hastings sampling to generate realizations from the appropriate distribution, where we discard 99 samples before every one kept, which both provides a burn-in effect and reduces the serial correlation between 
samples. 

The results in Figures~\ref{Fig:HermRand} and~\ref{Fig:LegRand} identify a {\it phase transition},~\cite{Donoho09b}, in the ability of $\ell_1$-minimization to recover $\bm c$. For a number of solution realizations, given by $N/P$, the method succeeds -- with probability one  -- in reconstructing solutions with high enough sparsity, given by small $s/N$, and fails to do so for low sparsity. Between these two phases, the method recovers the solution with probability smaller than one. Here, we observe differentiation in the quality of solution recovery in the transition region based on how $\bm{\Psi}$ is sampled. In particular, we highlight the following notable observations: for the high order case $(d,p)=(2,30)$, the standard Hermite sampling performs poorly as compared to the uniform sampling, for the high-dimensional case $(d,p)=(30,2)$, the standard Legendre sampling is much better than the Chebyshev sampling, and for the moderate values of $(d,p)=(5,5)$ the two sampling methods lead to similar 
performance. In all cases, the MCMC sampling leads to recovery that is similar to those of the other two sampling strategies or provides considerable improvements.\\

\noindent{\bf Remark:} Though our results following from~\cite{CandesPlan} do not necessarily imply uniform recovery over all functions of a certain sparsity, the results in this section are appropriately interpreted in the context of uniform recovery. For a more detailed definition of uniform recovery, we refer the interested reader to \cite{Rauhut10}.

\begin{figure}[h!]
\centering
\includegraphics[trim=47mm 0mm 0mm 10mm, clip, scale=0.5]{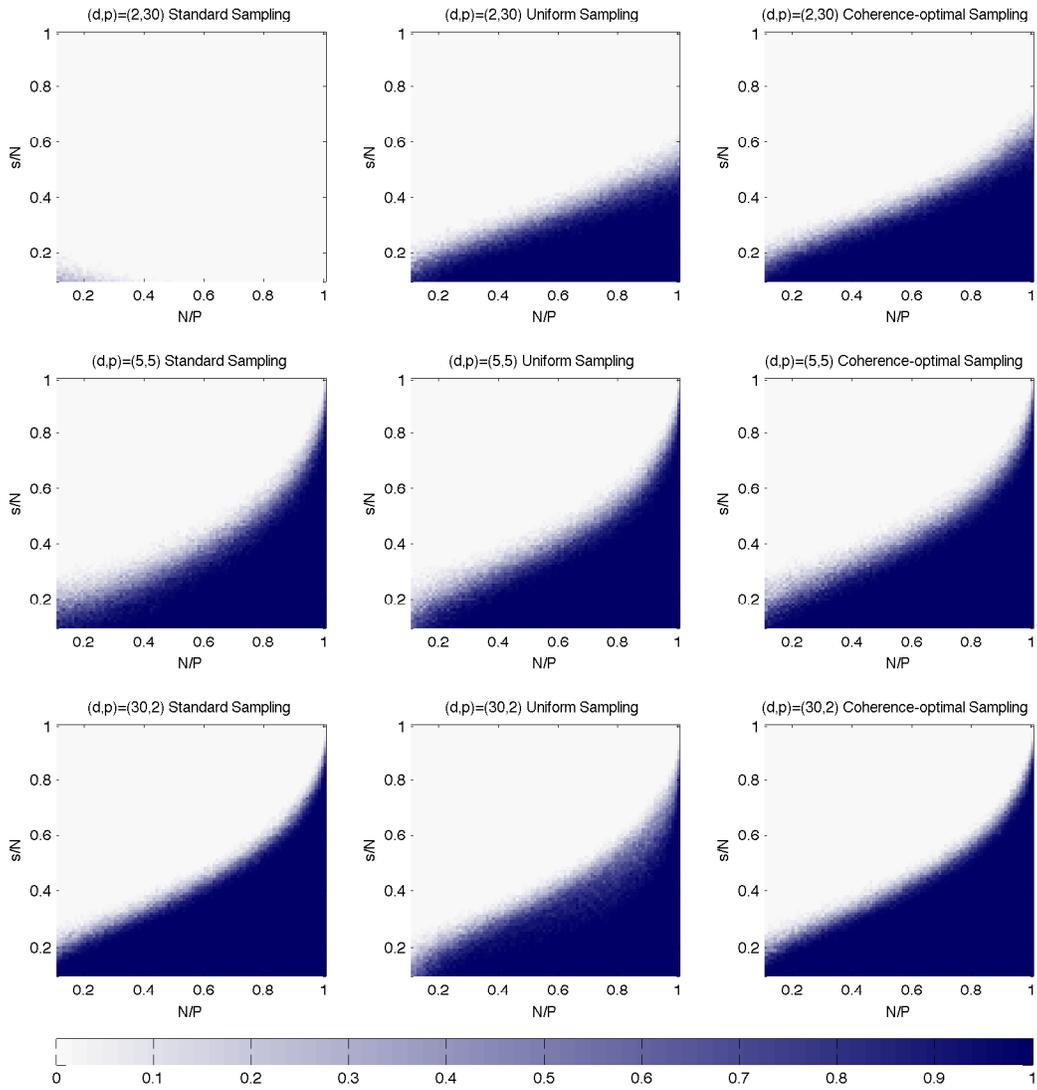}
\caption{Hermite Recovery Phase Diagrams: The rows correspond to differing dimension and total order while the columns correspond to the different sampling schemes. The color of each square represents a probability of successful function recovery.}
\label{Fig:HermRand}
\end{figure}

\begin{figure}[h!]
\centering
\includegraphics[trim=47mm 0mm 0mm 10mm, clip, scale=0.5]{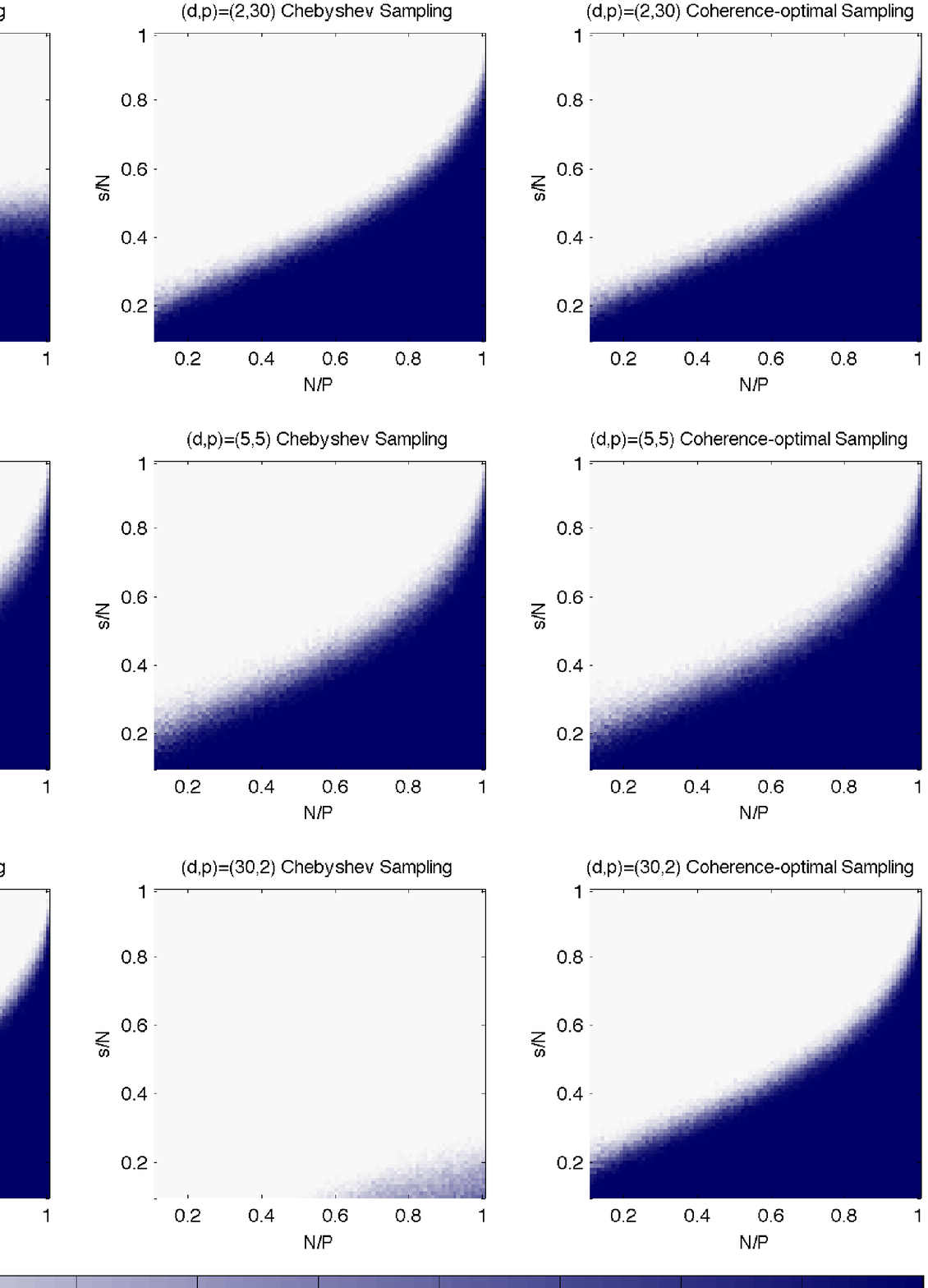}
\caption{Legendre Recovery Phase Diagrams: The rows correspond to differing dimension and total order while the columns correspond to the different sampling schemes. The color of each square represents a probability of successful function recovery.}
\label{Fig:LegRand}
\end{figure}

\subsection{An Elliptic PDE with Random Input}
\label{subsec:Elliptic}

As an application of Legendre PC expansions, we next consider the solution of the linear elliptic PDE 
\begin{eqnarray}
\label{Eqn:EllOperator}
\nabla\cdot(a(\bm{x},\bm{\Xi})\nabla u(\bm{x},\bm{\Xi}))&=&1,\quad \bm{x}\in\mathcal{D},\nonumber\\
u(\bm{x},\bm{\Xi})&=&0 ,\quad \bm{x}\in\partial\mathcal{D},\nonumber
\end{eqnarray}
on the unit square $\mathcal{D}= (0,1)\times(0,1)$ with boundary $\partial\mathcal{D}$. The diffusion coefficient $a$ is considered random and is modeled by 
\begin{align}
\label{eqn:gaussian_field}
a(\bm{x},\bm{\Xi})=a_0 +\sigma_a\sum_{k=1}^{d}\sqrt{\zeta_k}\varphi_{k}(\bm{x})\Xi_{k},
\end{align}
in which the random variables $\{\Xi_{k}\}_{k=1}^d$, $d=20$, are independent draws from the U([-1,1]) distribution, and we choose $a_0 = 0.1$ and $\sigma_a=0.017$. In (\ref{eqn:gaussian_field}), $\{\zeta_k\}_{k=1}^d$ are the $d$ largest eigenvalues associated with $\{\varphi_k\}_{k=1}^d$, the $L_2([0,1]^2)$-orthonormalized eigenfunctions of
\begin{align}
\label{eqn:GaussCov}
C_{aa}(\bm{x},\bm{y}) = \exp{\left[-\frac{(x_1-y_1)^2}{l_{1}^2}-\frac{(x_2-y_2)^2}{l_{2}^2}\right]}
\end{align}
with correlation lengths $l_{1}=0.8, l_{2}=0.1$ in the spatial dimensions. Given these choices of parameters, the model in (\ref{eqn:gaussian_field}) leads to strictly positive realizations of $a$.

For any realization of $\bm{\Xi}$, we use the finite element solver FEniCS~\cite{FEniCS} to compute an approximate solution $u(\bm{\Xi})=u((0.5,0.5),\bm{\Xi})$.

To identify $u(\bm{\Xi})$ as a function of the random inputs $\bm{\Xi}$, we use a Legendre PC expansion of total order $p=4$, which for this $d=20$ stochastic dimensional problem yields $P=10,626$ basis functions. We note that the root-mean-squared error is considered here as the primary measure of recovery.

We investigate the ability to recover $u(\bm{\Xi})$ via (\ref{eqn:constrained}), using each of the three sampling schemes considered for Legendre polynomials. For this elliptic problem we further improve the quality of the MCMC sampling through an initial burn-in of 1,000 discarded samples~\cite{MCMCBook}. We provide bootstrapped estimates of the various moment based measures from a pool of samples generated beforehand. Specifically, samples for each realization are drawn from a pool of 50,000 previously generated samples, which are used to calculate bootstrap estimates of averages and standard deviations.

To identify the solution of (\ref{eqn:constrained}) we use the SPGL1 package,~\cite{SPGL1,SPGL2}, with a truncation error in (\ref{eqn:constrained}), denoted by $\delta$, and determined for each set of samples by two-fold, also known as hold-out, cross-validation,~\cite{CrossValidation}. Specifically, we calculate this $\delta$ from $N$, an even number of available samples, by splitting the available samples into two equally sized sets, one a training set, and the other a validation set. This process, as we have implemented it, is summarized by the following algorithm,

\begin{enumerate}
\item For a number of $\delta$, construct solutions, $\bm{c}_\delta$, from the training set and the solution of (\ref{eqn:constrained}). We use a set of potential $\delta$ defined by $10^{-(-1:0.05:5)}$.
\item For each $\bm{c}_\delta$ use the validation set to identify the reconstruction error $\epsilon^{(1)}_\delta:= \|\bm{\Psi}\bm{c}_\delta-\bm{u}\|_2$.
\item Repeat with the training and validation sets swapped to attained $\epsilon^{(2)}_\delta$.
\item Identify the $\delta_0$ that minimizes $\epsilon^{(1)}_\delta+\epsilon^{(2)}_\delta$.
\item Set the truncation error to $\delta_\star=\sqrt{2}\delta_0$, and identify a solution vector via the combined $N$ samples from both the training and validation sets.
\end{enumerate}

We utilize this method of cross-validation to calibrate the truncation error for each realized sample of the calculated solution to (\ref{eqn:constrained}). Here, a lower cross-validated truncation error suggests a computed solution vector with a more accurate recovery. 

In Figure~\ref{Fig:EllipticSD} we see plots of computed moments for the distribution of the relative root-mean-squared error between the computed and reference solutions obtained from 100 independent replications for each sample size, $N$. In addition, Figure~\ref{Fig:EllipticTol} presents similar plots for the truncation error computed with each sampling. We note here that the cross-validated computation of $\delta$ provides an available estimate for anticipated root-mean-squared error for additional independent samples.

We note the standard and coherence-optimal sampling offer significant improvements over asymptotic, i.e., Chebyshev, sampling using similar sample sizes, $N$, both in terms of accuracy and robustness to differing realized samples. These observations are compatible with the theoretical results of Section \ref{sec:sampling} demonstrating a smaller coherence for the uniform sampling -- as compared to Chebyshev sampling -- for the case of $d>p$. The coherence-optimal sampling by construction leads to smallest coherence. We also notice that at particularly low sample-sizes any given sampling method prefers to recover a particular but ultimately poor approximation. As the number of samples increases the recovery can improve but the variability in the solution recovery will appear to increase first as this preferred solution is recovered less frequently.

\begin{figure}[h!]
\hspace{-1.2cm}
\begin{minipage}{.5\textwidth}
\centering
\includegraphics[scale=0.3]{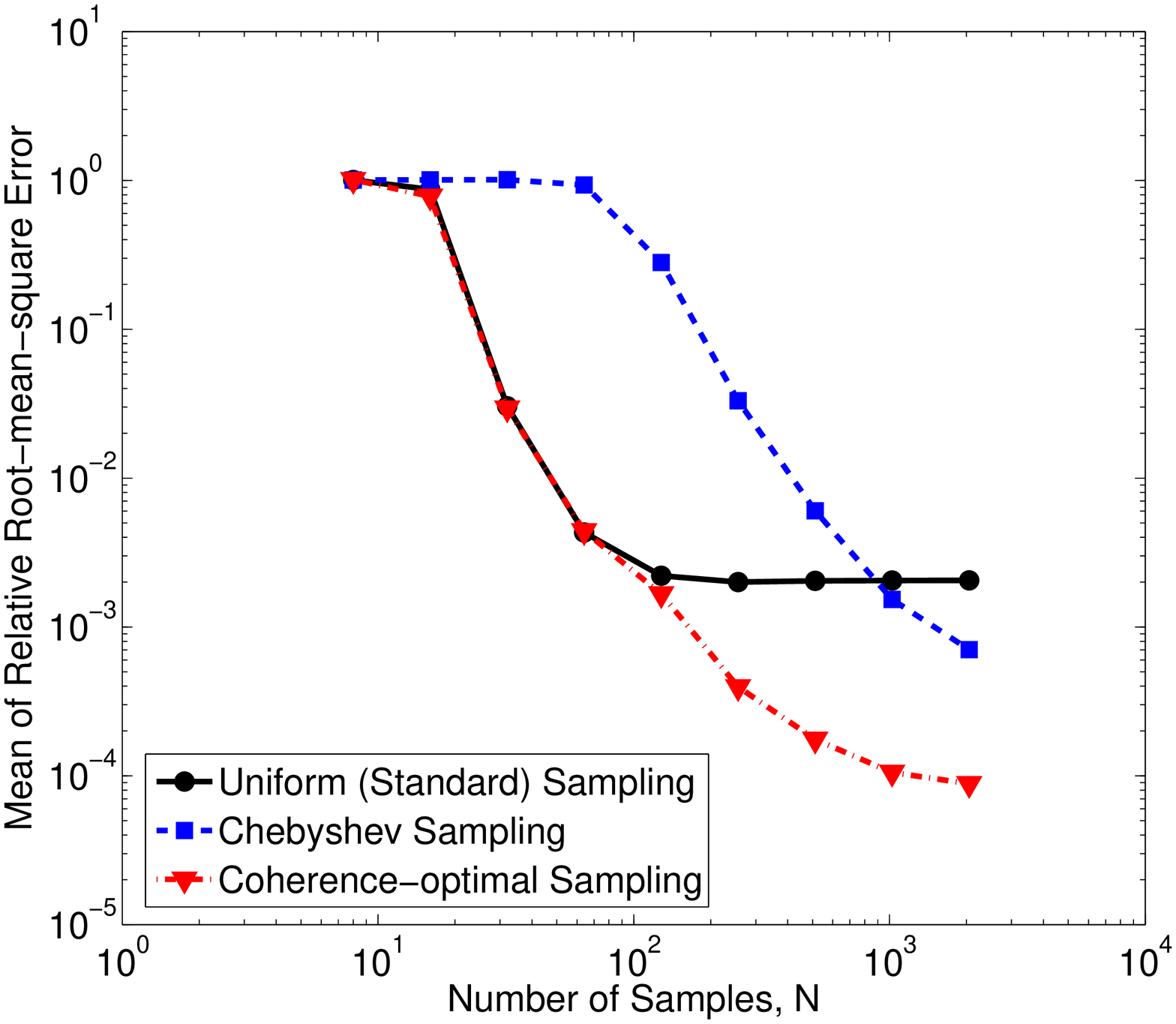}
\end{minipage}
\hspace{.6cm}
\begin{minipage}{.5\textwidth}
\centering
\includegraphics[scale=0.3]{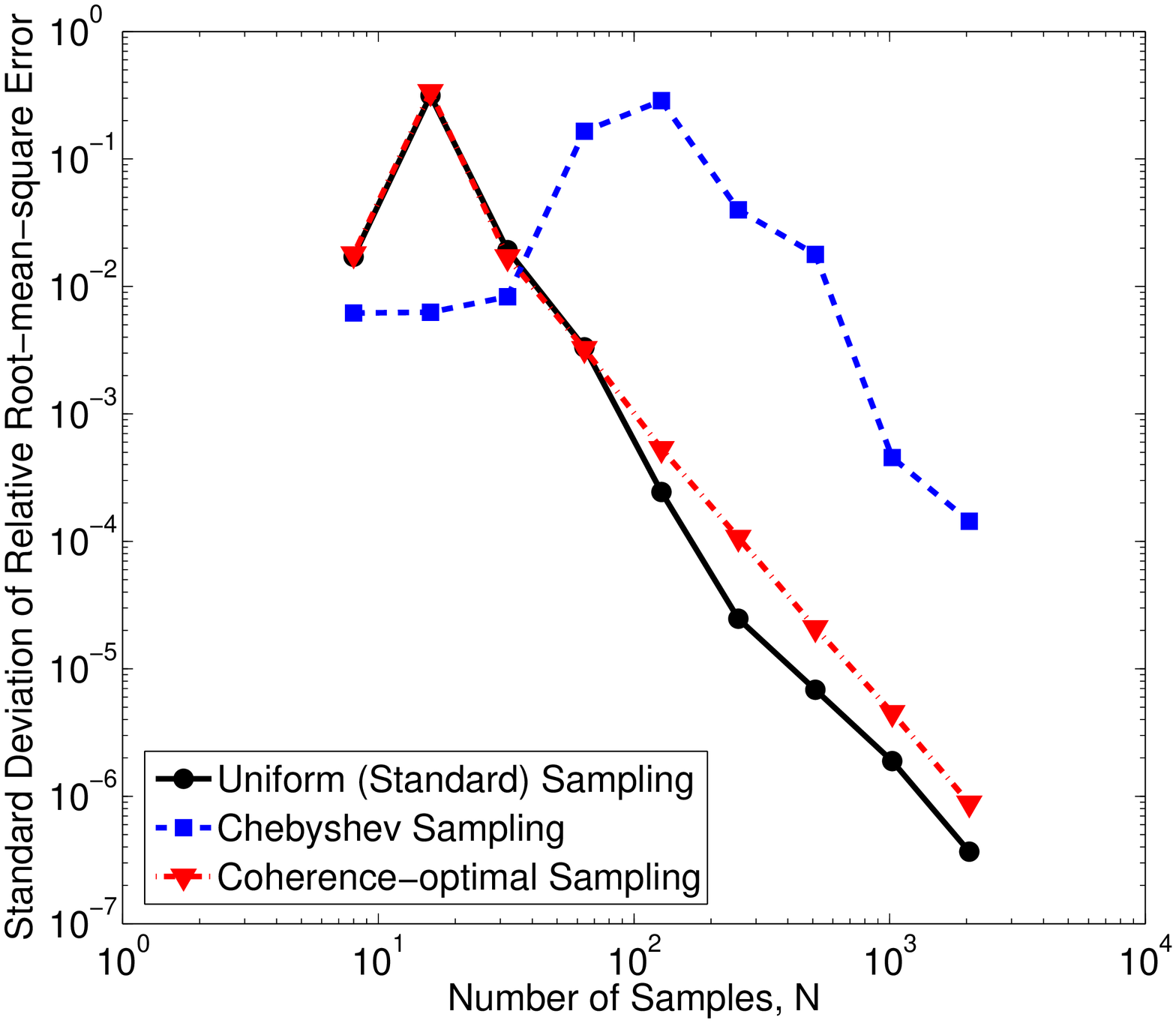}
\end{minipage}
\caption{Plots for the moments of root-mean-squared error for independent residuals for the various sampling methods as a function of the number of samples.}
\label{Fig:EllipticSD}
\end{figure}

\begin{figure}[h!]
\hspace{-1.2cm}
\begin{minipage}{.5\textwidth}
\centering
\includegraphics[scale=0.3]{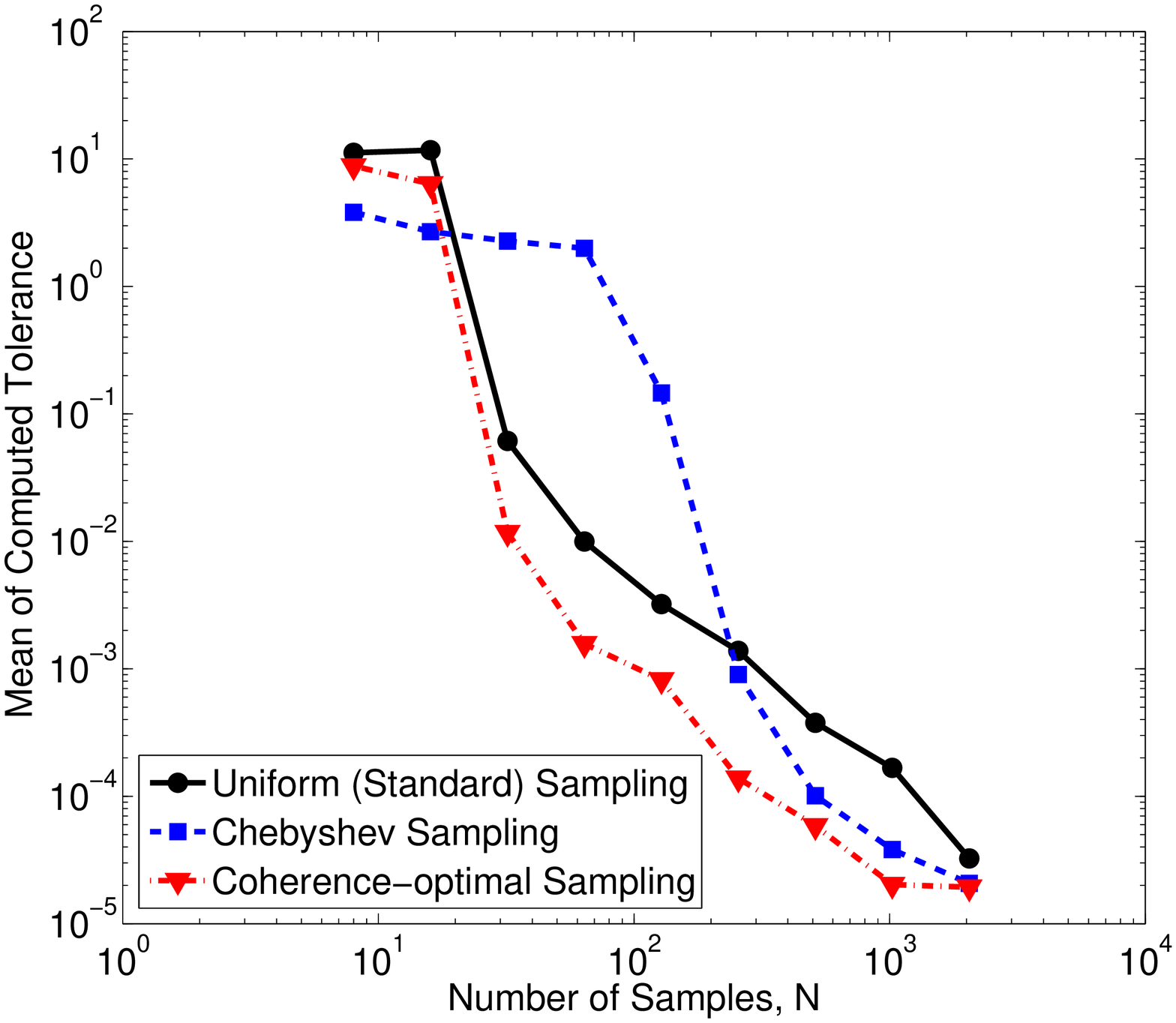}
\end{minipage}
\hspace{.6cm}
\begin{minipage}{.5\textwidth}
\centering
\includegraphics[scale=0.3]{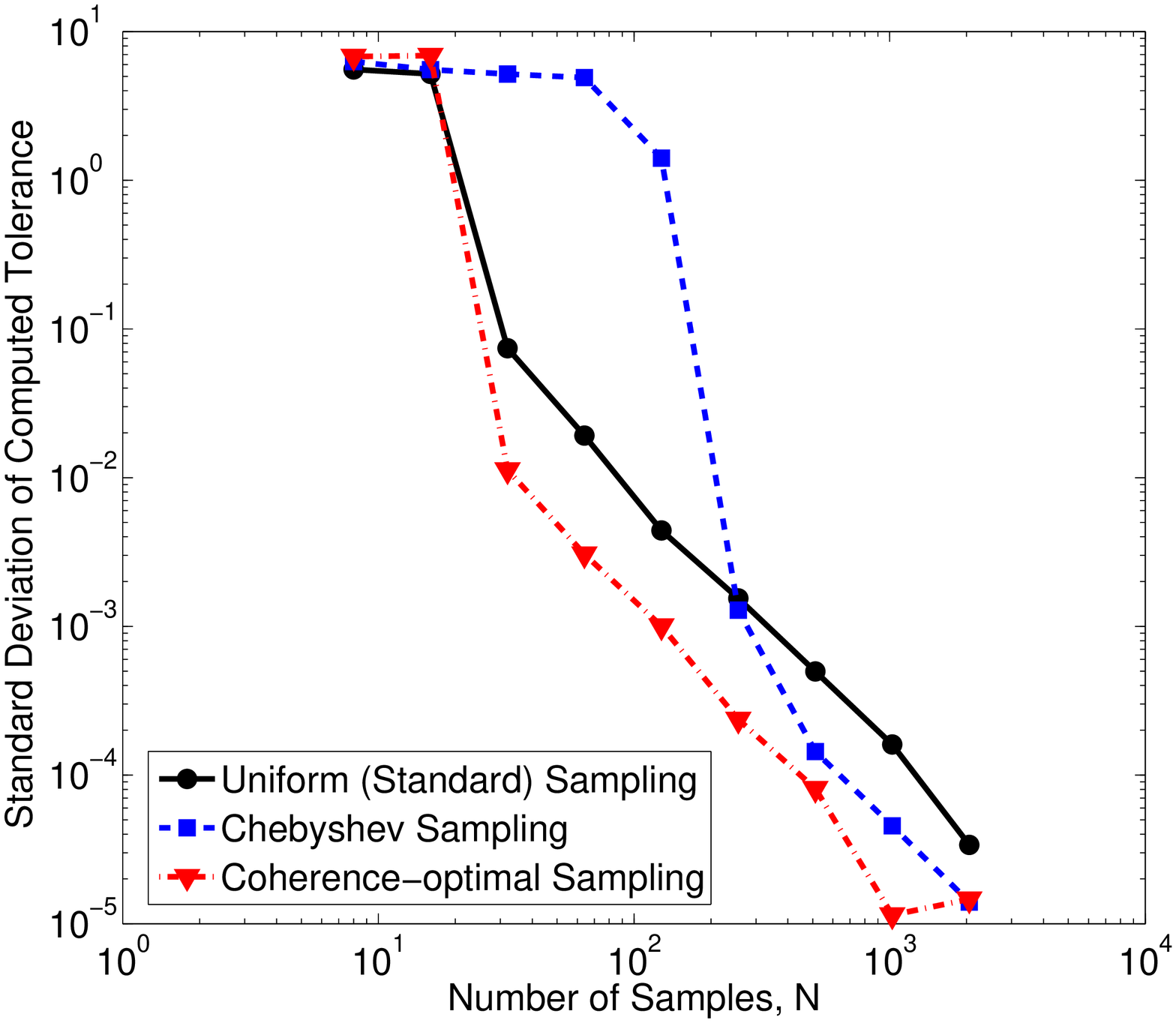}
\end{minipage}
\caption{Plots for the moments of cross-validated estimates of tolerance for the various sampling methods as a function of the number of samples.}
\label{Fig:EllipticTol}
\end{figure}

%
%
\subsection{Surface Reaction Model}
\label{subsec:EllipticOperator}

Another problem of interest in this work is to quantify the uncertainty in the solution $\rho$ of the non-linear evolution equation  
\begin{align}
\label{eqn:reaction}
\left\{\begin{array}{l}\frac{d\rho}{dt} = \alpha(1-\rho) - \gamma\rho - \kappa(1-\rho)^2\rho, \\ 
\rho(t=0) = 0.9,\end{array}\right.
\end{align}
modeling the surface coverage of certain chemical species, as examined in \cite{Makeev02,Lemaitre04b}. We consider uncertainty in the adsorption, $\alpha$, and desorption, $\gamma$, coefficients, and model them as shifted log-normal variables. Specifically, we assume
\begin{align*}
\alpha &= 0.1 + \exp(0.05\ \Xi_1),\\
\gamma &= 0.001 + 0.01\exp(0.05\ \Xi_2),
\end{align*}
where $\Xi_1, \Xi_2$ are independent standard normal random variables; hence, the dimension of our random input is $d=2$. The reaction rate constant $\kappa$ in (\ref{eqn:reaction}) is assumed to be deterministic and is set to $\kappa = 10$. 

Our QoI is $\rho_c:=\rho(t=4,\Xi_1,\Xi_2)$, and to approximate this, we consider a Hermite PC expansion of total order $p=32$, giving $P=561$ basis functions. This high-order approximation is necessary due to the large gradient of $\rho_c$ in terms of the random variables, as evidenced by the relatively slow decay of coefficients in the reference solution presented in Figure~\ref{Fig:RefSolODE}. This is computed using Gauss-Hermite quadrature approximation of the PC coefficients.  

We utilize the same computational process as in Section~\ref{subsec:Elliptic} to identify approximate solutions. In Figure~\ref{Fig:ODERMSE}, we see plots of moments for the relative root-mean-squared error -- between the reference and $\ell_1$-minimization solutions -- as a function of the number of samples, $N$. These moments are obtained from 200 independent replications for each $N$. We find that the standard sampling fails to converge, while the {\color{black}uniform} and coherence-optimal samplings lead to converged solution as $N$ is increased. Figure~\ref{Fig:ODETol} presents plots for the truncation error computed with each sampling via cross-validation.  One interesting fact to notice is that recovery for standard sampling appears to get worse for larger sample sizes. This may be an effect of the poor numerical conditioning of high order $p=30$ Hermite polynomials under standard sampling, where very rare events with very large realized $|\psi_k(\bm{\xi})|$ are necessary to capture the orthogonality of the polynomials. It further affirms the results of Figure~\ref{Fig:HermCoh} and Theorem~\ref{thm:NatSampleCoherence}, that standard sampling of Hermite polynomials is not suited for high-order problems.

\begin{figure}[h!]
\centering
\includegraphics[scale=0.6]{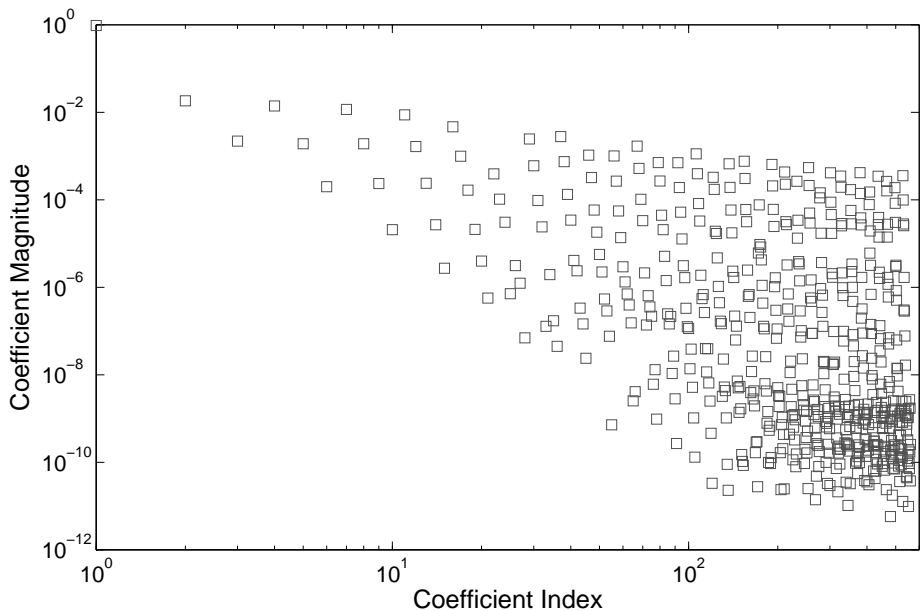}
\caption{PC coefficients of reference solution for the QoI $\rho_c:=\rho(t=4,\Xi_1,\Xi_2)$ in the surface reaction model.}
\label{Fig:RefSolODE}
\end{figure}

\begin{figure}[h!]
\hspace{-1.1cm}
\begin{minipage}{.5\textwidth}
\centering
\includegraphics[scale=0.32]{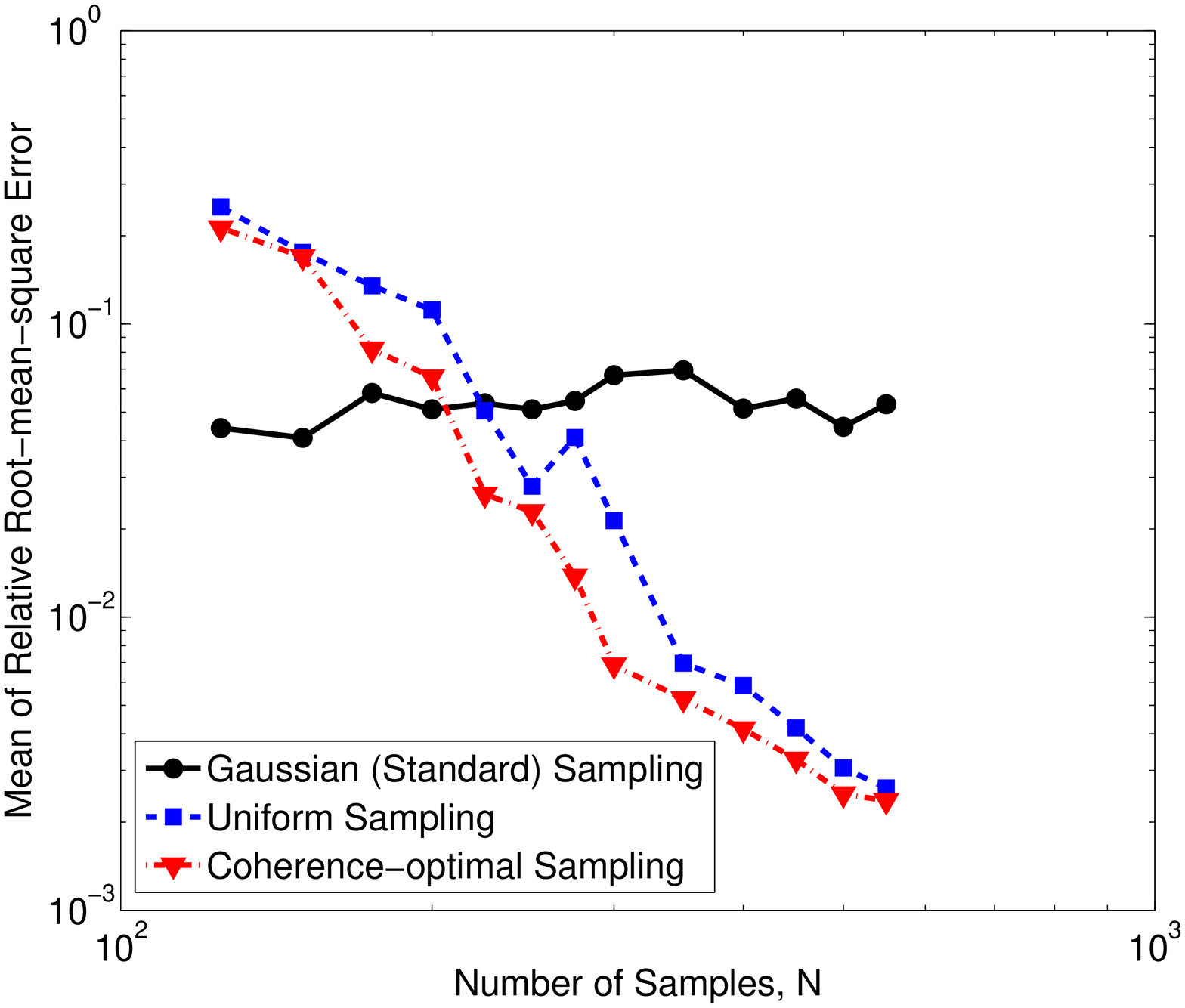}
\end{minipage}
\hspace{.6cm}
\begin{minipage}{.5\textwidth}
\centering
\includegraphics[scale=0.32]{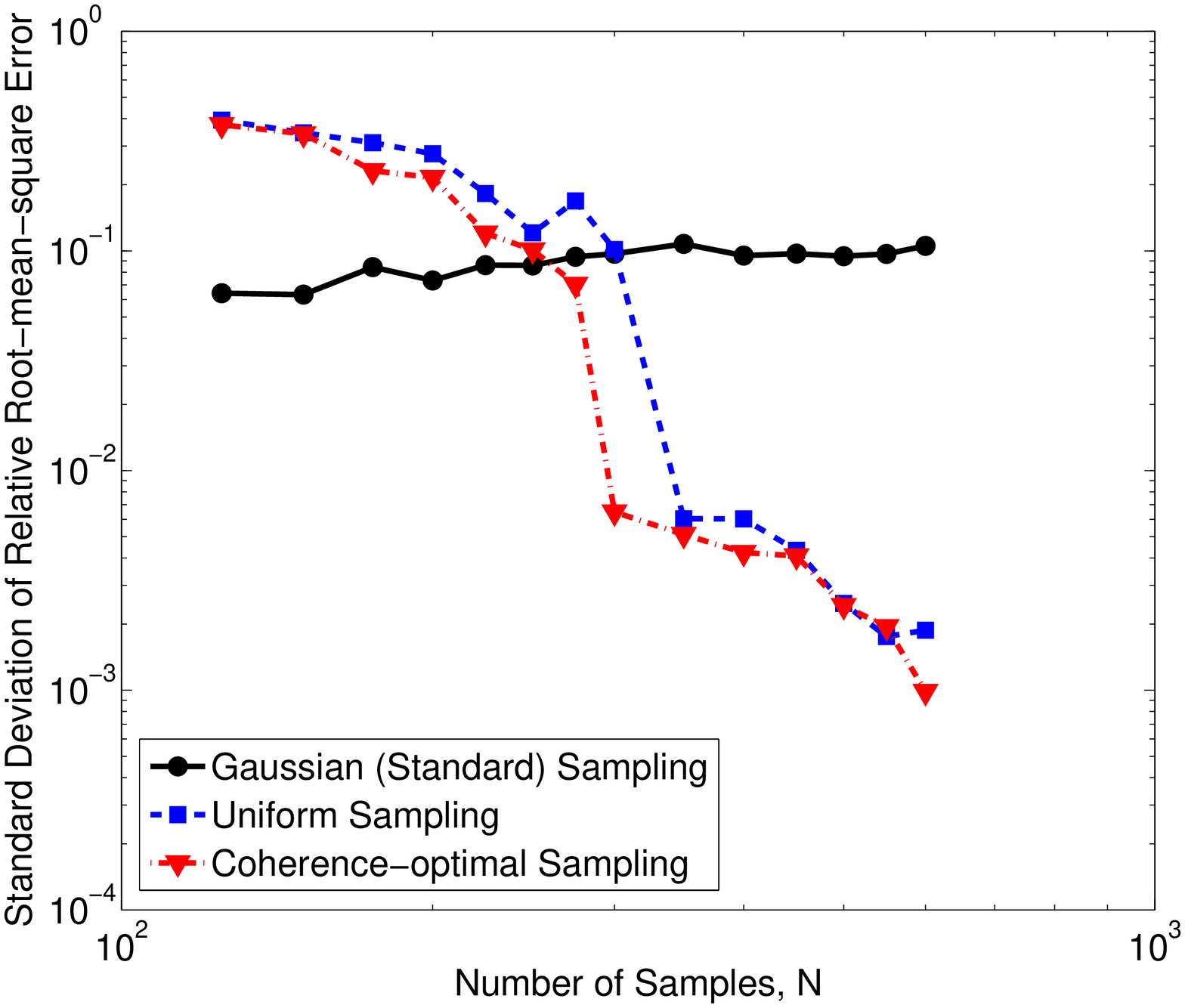}
\end{minipage}
\caption{Moments of root-mean-squared error between the reference and $\ell_1$-minimization solutions for the various sampling methods as a function of the number of samples.}
\label{Fig:ODERMSE}
\end{figure}

\begin{figure}[h!]
\hspace{-1.1cm}
\begin{minipage}{.5\textwidth}
\centering
\includegraphics[scale=0.32]{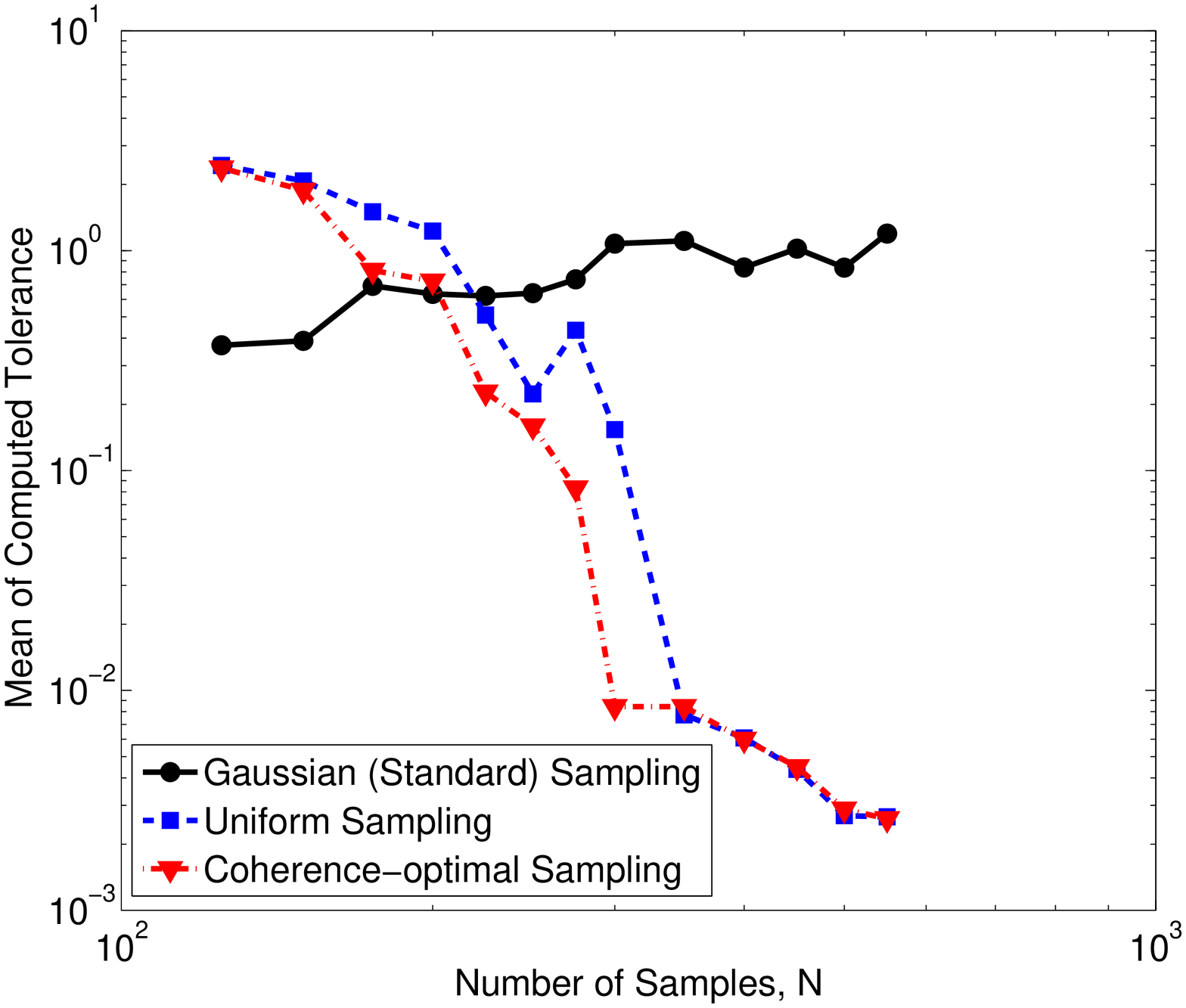}
\end{minipage}
\hspace{.6cm}
\begin{minipage}{.5\textwidth}
\centering
\includegraphics[scale=0.32]{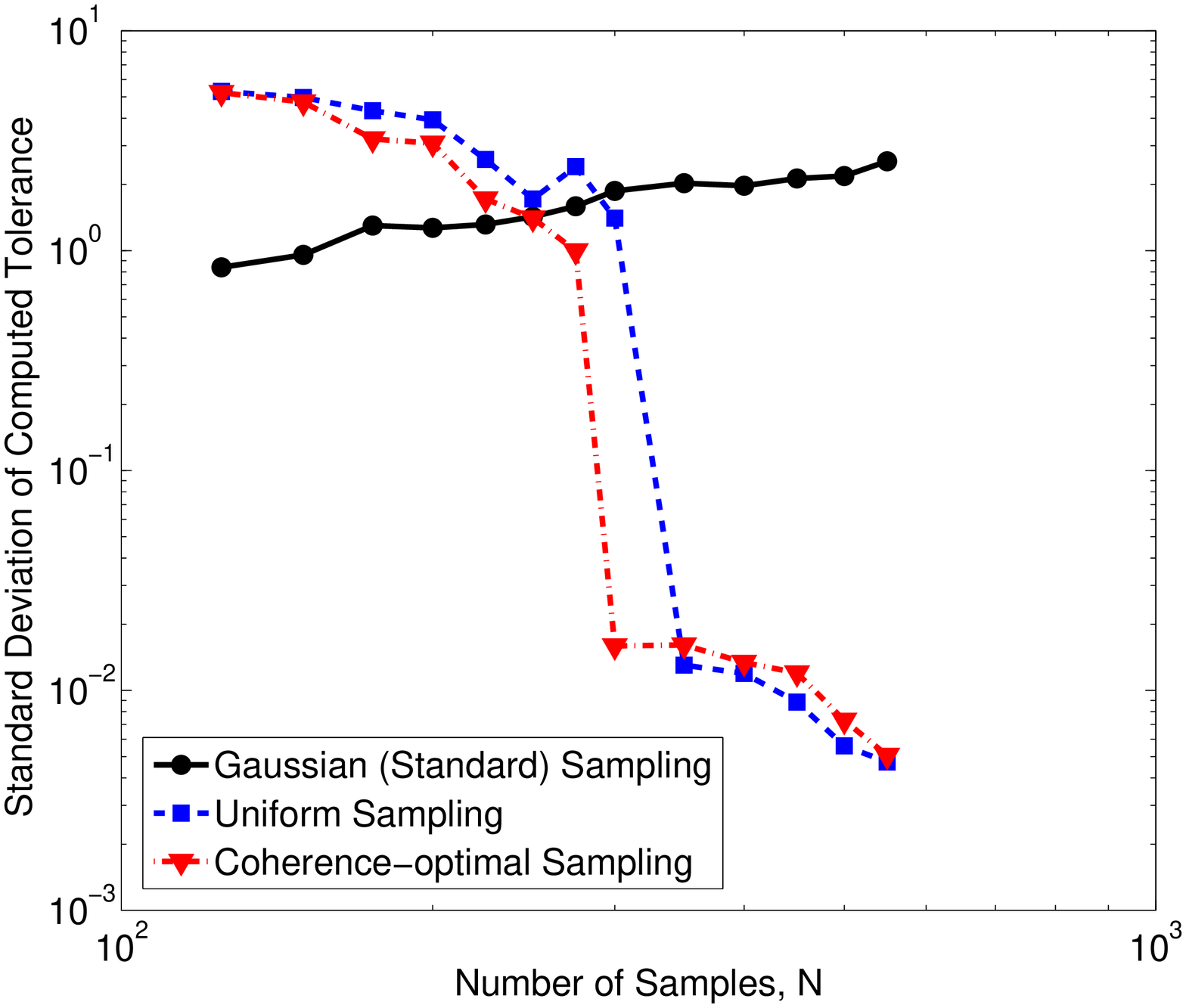}
\end{minipage}
\caption{Plots for the moments of cross-validated estimates of tolerance for the various sampling methods as a function of the number of samples.}
\label{Fig:ODETol}
\end{figure}

\section{\texorpdfstring{Proofs}{Proofs}}
\label{sec:Proofs}

Here we present proofs for the theorems in Section~\ref{sec:sampling}.  These proofs, except for that of Theorem~\ref{thm:LowPTransformedSamples}, rely on an analysis of the appropriate orthonormal polynomials. We first work toward proofs for Theorems~\ref{thm:NatSampleCoherence} and~\ref{thm:TransformedSamples}, which rely on understanding Hermite polynomials. To do so we require a few technical Lemmas concerning broad behavior of the polynomials asymptotically in order.

This analysis is focused on three domains for one-dimensional polynomials. The first sampling region coincides nearly with the so-named oscillatory region of $\psi_{p}(\xi)$,~\cite{AskeyWainger}, within which all zeros of $\psi_k(\xi)$ are found for $k\le p$. A second region, referred to as the monotonic region,~\cite{AskeyWainger}, is the complementary region where polynomials tend to increase monotonically in magnitude, and in this region we focus on bounding the extreme values of $\psi_k(\xi)$. The third region of importance is the boundary between the monotonic and oscillatory regions, referred to as the boundary region.

For multidimensional Hermite polynomials, we identify a domain $\mathcal{S}$ which fully contains the multidimensional analogue to the oscillatory and boundary regions, and partially contains the monotonic region. The size of the monotonic region included is determined so as to satisfy the conditions of (\ref{Eqn:CoherenceUnbounded}), while admitting a useful bound on the extreme values of $|\psi_k(\bm{\xi})|$ for $\bm{\xi}\in\mathcal{S}$. The method for our selection of $\mathcal{S}$ is to include $\bm{\xi}$ corresponding to the largest values of the density function, $f(\bm{\xi})$, until $\mathcal{S}$ is verified to satisfy (\ref{Eqn:CoherenceUnbounded}). In the case of Hermite polynomials, this heuristic is justified as Hermite polynomials tend to inversely relate with $f(\bm{\xi})$ for $\bm{\xi}$ within the monotonic region,~\cite{Szego}. The selection of $\mathcal{S}$ determines our radius for sampling uniformly from the $d$-dimensional ball, and we will show that this involves taking $\mathcal{S}=\{\bm{
\xi}:\|\bm{\xi}\|_2\le r_p\}$ for an $r_p$ that grows asymptotically like $2\sqrt{p}$.

\subsection{Key Hermite Lemmas}
\label{subsec:keyHermite}

For convenience with the cited literature, we prove our results using the orthonormalized physicists' Hermite polynomials (orthonormal polynomials with respect to $f(\bm{\xi}):=\pi^{-d/2}\exp(-\|\bm{\xi}\|^2)$). We note that our results in Section~\ref{sec:sampling} are in terms of the probabilists' polynomials (orthogonal with respect to $f(\bm{\xi}):=(2\pi)^{-d/2}\exp(-\|\bm{\xi}\|^2/2)$), but the two sets are related as follows. If $\{\psi_k(\xi)\}$ denotes the orthonormalized physicists' polynomials and $\{\psi^{\prime}_k(\xi)\}$ represents the orthonormalized probabilists' polynomials, then for each $k$, $\psi_{k}(\sqrt{2}\xi)=\psi^{\prime}_k(\xi)$. We point the reader to Section 5.5 of~\cite{Szego} for a derivation of this key relation. The effect on the results of the proof is that the probabilists' polynomials require a sampling radius that is $\sqrt{2}$ times larger than that for sampling the physicists' polynomials. This radial effect does not effect the volume of the points in the interior, 
particularly as seen in Theorem~\ref{thm:TransformedSamples}, as the radius change is cancelled out by the change in the normalizing constant for the distribution ($\pi^{-d/2}$ vs. $(2\pi)^{-d/2}$).

We bound (\ref{Eqn:CoherenceUnbounded}) for the $d$-dimensional Hermite polynomials as follows. Let $\bm{\xi}$ be a $d\times 1$ vector and $\bm{k}$ be a $d\times 1$ multi-index. In this framework $\psi_{\bm{k}}(\bm{\xi})$ is an orthonormal polynomial whose order in the $i$th dimension is given by $k_i := \bm{k}(i)$, and whose total order is at most $p$. As the total order is at most $p$, $\|\bm{k}\|_1\le p$, and as the weight function is formed by a tensor product of one dimensional weight functions, $\psi_{\bm{k}}$ is a tensor product of univariate orthogonal polynomials. In this way the bounds in arbitrary dimension are tensor products of one-dimensional bounds, which are more easily derived.

As mentioned previously, the behavior of the Hermite polynomials in the monotonic region, and the radially symmetric concentration of the weight function $\pi^{-1/2}\exp(-\|\bm{\xi}\|_2^2)$ suggests the candidate set $\mathcal{S}:=\{\|\bm{\xi}\|_2\le r_p\}$ with $r_p$ to satisfy the conditions of (\ref{Eqn:CoherenceUnbounded}). We recall that the minimum over admissible $\mathcal{S}$ yields a coherence parameter less than any given choice of $\mathcal{S}$, so that this selection of $\mathcal{S}$ leads to an upper bound on a minimal $\mu(\bm{Y})$.

Being of classical and modern importance, several classes of one-dimensional orthogonal polynomials (e.g., Hermite, Jacobi, Legendre, Laguerre) have received much analysis and key results are available in the literature,~\cite{RauhutWard,Szego,Hermite1,AskeyWainger,LagAsym,Jacobi}.

In particular, for our interest in Hermite polynomials, a direct consequence of bounds from~\cite{AskeyWainger} gives the bounds in Table~\ref{Tab:Bounds} for some positive $C,\gamma$, and $n:=2k+1$. The key conclusion is that we may bound $\exp(-\xi^2/2)\psi_k(\xi)$ for $\xi$ in each of these regions. 
\begin{table}
\center
\begin{tabular}{|l|l|}
\hline
Range for $\xi$ & Bound for $|\psi_k(\xi)|$\\
\hline
$0\le |\xi|\le n^{1/2}-n^{-1/6}$ & $Cn^{-1/8}(n^{1/2}-|\xi|)^{-1/4}\exp(\xi^2/2)$\\
\hline
$n^{1/2}-n^{-1/6}\le |\xi|\le n^{1/2}+n^{-1/6}$ & $Cn^{-1/12}\exp(\xi^2/2)$\\
\hline
\end{tabular}
\caption{Bounds in the {\color{black}oscillatory} and boundary regions of Hermite polynomials from~\cite{AskeyWainger}. Here, $C$ is some positive constant and $n:=2k+1$.}
\label{Tab:Bounds}
\end{table}

The bounds in Table~\ref{Tab:Bounds} are sufficient for both our uses within the oscillatory region and the boundary of the oscillatory and monotonic region. We derive a bound within the monotonic region using results in~\cite{Hermite1}. We first summarize the needed results as follows. Let $\sigma_k(\xi):=\sqrt{\xi^2-2k}$ for $|\xi|\ge \sqrt{2k}+\epsilon_k$ where $\epsilon_k\rightarrow 0$ as $k\rightarrow\infty$. We note that our analysis does not address how rapidly we may take $\epsilon_k$ to $0$, and for our purposes it is more convenient to redefine $\epsilon_k$ such that $|\xi|\ge\sqrt{(2+\epsilon_k)k+1}$, again letting $\epsilon_k\rightarrow 0$. This lack of effective analysis for $\epsilon_k$ implies a lack of effective analysis for derived quantities, and all results are guaranteed to hold in an asymptotic sense without an analysis as to how rapidly convergence occurs. We do refer the reader to~\cite{Hermite1} for some analysis of how $\epsilon_k$ may be taken to zero, specifically in a worst case, $\epsilon_k=O(k^{-1/6})$. As a matter of notation, while $\sigma_k(\xi)$ depends on both $k$ and $\xi$, in what follows we suppress the dependence on $\xi$. Following~\cite{Hermite1}, we may approximate $\psi_k(\xi)$ when $|\xi|\ge\sqrt{(2+\epsilon_k)k+1}$ by 
\begin{align}
\nonumber
\frac{c^{\prime}_k}{C_k}\exp\left(\frac{\xi^2-\sigma_k \xi-k}{2}\right)(\sigma_k+\xi)^k\sqrt{\frac{1}{2}\left(1+\frac{\xi}{\sigma_k}\right)} &\le \psi_k(\xi);\\
\label{eqn:BetterAsym}
\frac{c_k}{C_k}\exp\left(\frac{\xi^2-\sigma_k \xi-k}{2}\right)(\sigma_k+\xi)^k\sqrt{\frac{1}{2}\left(1+\frac{\xi}{\sigma_k}\right)} &\ge \psi_k(\xi),
\end{align}
where both $c^{\prime}_k, c_k\rightarrow 1$ as $k\rightarrow \infty$, and $C_k=\sqrt{2^{k}k!}$ is the appropriate constant so that the $\{\psi_k\}$ are orthonormal with regards to the weight $f(\xi)=\pi^{-1/2}\exp(-\xi^2)$.  For smaller $|\xi|$ the polynomials are effectively oscillatory and more technically troublesome to work with. Thankfully, as we are more concerned with approximating key integrals where we know the value ($1$ or $0$ by orthonormality) over the real line, we do not need to delve closely into the analysis for small $|\xi|$, and understanding the behavior for large $|\xi|$ is sufficient.

The key technical results to bound $\psi_k$ in the monotonic region are presented in the following lemma, where the motivating idea is to show that the polynomial $\psi_k(\xi)$ is tightly bounded by an envelope with a well behaved exponential parameter, denoted by $\eta_k(\xi)$. Due to the length of the proof we delay the proof to Appendix A.
\begin{lem}
\label{lem:PolyBeh}
Let $C_k$ and $\sigma_k$ be as in (\ref{eqn:BetterAsym}), and define the function $\eta_k(\xi)$ implicitly by,
\begin{align}
\label{eqn:ExpApproximation}
\frac{1}{C_k}\exp\left(\frac{\xi^2-\sigma_k \xi-k}{2}\right)(\sigma_k+\xi)^k\sqrt{\frac{1}{2}\left(1+\frac{\xi}{\sigma_k}\right)}&=\exp\left(\eta_k(\xi)\xi^2\right).
\end{align}
That is we approximate $|\psi_k(\xi)|$ by $\exp\left(\eta_k(\xi)\xi^2\right)$ with the exponent $\eta_k(\xi)$ implicitly defined by the approximation in (\ref{eqn:BetterAsym}).

It follows that 
\begin{enumerate}
\item For $\epsilon> 0$, $\mathop{\lim}\limits_{k\rightarrow\infty}\eta_k(\sqrt{(2+\epsilon)k+1})\rightarrow \frac{1}{2}-\frac{\log(2)}{2(2+\epsilon)}$.
\item For a sequence of $\epsilon_k>0$ such that $\epsilon_k\rightarrow 0$ as $k\rightarrow\infty$, and for $\xi_1>\xi_0\ge\sqrt{(2+\epsilon_k)k+1}$, $\eta_k(\xi_1)<\eta_k(\xi_0)$.
\item For a sequence of $\epsilon_k$ such that $\epsilon_k\rightarrow 0$, some finite $K$ and $k_1\ge K$, $k_0<k_1$, and for $\xi\ge\sqrt{(2+\epsilon_{k_1})k_1+1}$, it follows that $\eta_{k_0}(\xi)<\eta_{k_1}(\xi)$.
\end{enumerate}
\end{lem}

Using Lemma~\ref{lem:PolyBeh}, we show the following result which is useful for a direct bound on the coherence parameter.
\begin{lem}
\label{lem:CohBounds}
For some choice of $\epsilon_p$ such that $\epsilon_p\rightarrow 0$, and $p\ge p_0$ for some $p_0$ it follows that for $r_p\ge \sqrt{(2+\epsilon_p)p+1}$, %
\begin{align}
\label{eqn:CohBound1}
\mathop{\sup}\limits_{k\le p}\int_{|\xi|>r_p}\psi^2_k(\xi)\frac{e^{-\xi^2}}{\sqrt{\pi}}d\xi&\le\frac{(1+\delta_p)\mbox{erfc}(\sqrt{(1-2\eta_p(r_p))r_p^2})}{\sqrt{1-2\eta_p(r_p)}},
\end{align}
where $\mbox{erfc}(\cdot)$ is the complement to the error function and $\delta_p\rightarrow 0$.
Considering multidimensional polynomials and letting $\delta_p\rightarrow 0$,
\begin{align}
\label{eqn:CohBound2}
\mathop{\sup}\limits_{\substack{\|\bm{\xi}\|_2\le r_p\\ \|\bm{k}\|_1\le p}}|\psi_{\bm{k}}(\bm{\xi})|\le (1+\delta_p)\exp\left(\eta_p(r_p)r_p^2\right).
\end{align}
Here, $\eta_p$ is defined implicitly as in (\ref{eqn:ExpApproximation}), or equivalently, explicitly as in (\ref{eqn:ExpApp2}).
\end{lem}
\begin{proof}
To show the first point note from Lemma~\ref{lem:PolyBeh} that for $|\xi|\ge\sqrt{2p+1}$,
\begin{align*}
|\psi_k(\xi)|\le c_k\exp(\eta_k(\xi)\xi^2).
\end{align*}
By the second point of Lemma~\ref{lem:PolyBeh}, $\eta_k(\xi)$ decreases as $\xi$ increases, yielding the bound on the integral in (\ref{eqn:CohBound1}), and by the third point of that Lemma, $\eta_p(\xi)\ge \eta_k(\xi)$ for all $k\le p$. The $\delta_p$ accounts for the approximation in (\ref{eqn:BetterAsym}) being inaccurate for finite $p$, but we do not address how quickly $\delta_p$ converges to zero.

To show (\ref{eqn:CohBound2}) note that for column vectors $\bm{\eta}$ and $\bm{\xi}^2$ with coordinates representing coordinates of $\eta$ and $\xi^2$ in each dimension, $|\bm{\eta}^T\bm{\xi}^2|\le\|\bm{\eta}\|_\infty\|\bm{\xi}^2\|_1$, with equality holding if $\|\bm{\eta}\|_\infty$ is achieved at the one coordinate on which $\bm{\xi}^2$ is supported. Further noting that $\|\bm{\xi}^2\|_1=\|\bm{\xi}\|_2^2$, it follows from the third point of Lemma~\ref{lem:PolyBeh}, and hence for large enough $p$,
\begin{align*}
\mathop{\sup}\limits_{\substack{\|\bm{\xi}\|_2\le r_p\\ \|\bm{k}\|_1\le p}}|\psi_{\bm{k}}(\bm{\xi})|&\le (1+\delta_p)|\psi_{p}(r_p)|,
\end{align*}
where the bound on $\psi_p$ from (\ref{eqn:BetterAsym}) shows (\ref{eqn:CohBound2}).
\end{proof}

\subsection{Proof of Theorem~\ref{thm:NatSampleCoherence}}
\label{subsec:thmNSCH}
With Lemma \ref{lem:CohBounds} we are prepared to prove Theorem~\ref{thm:NatSampleCoherence} for the case of Hermite polynomials. Let $\mathcal{S} = \{\bm{\xi}:\|\bm{\xi}\|_2\le r_p\}$ where $r_p$ is as in Lemma~\ref{lem:CohBounds}, and we show that the conditions for (\ref{Eqn:CoherenceUnbounded}) are satisfied. Let the total number of polynomials be given by $P={p +d\choose d}$ where $p$ is the total order of the approximation and $d$ the number of dimensions. Recall that the number of samples from the orthogonal polynomial basis is $N$. We show that
\begin{align*}
\mathbb{P}(\mathcal{S}^c) = \mbox{erfc}(r_p)&<\frac{1}{NP};\\
\mathop{\sum}\limits_{k=1}^{P}\mathbb{E}\left[\psi^2_k(\bm{Z})\bm{1}_{\mathcal{S}^{c}}\right]\le \frac{P(1+\delta_p)\mbox{erfc}(\sqrt{(1-2\eta_p(r_p))r_{p}^2})}{\sqrt{1-2\eta_p(r_p)}}& < \frac{1}{20\sqrt{P}},
\end{align*}
where $\bm{Z}$ is normally distributed with variance $1/2$, and we recall that substituting $\bm{Z}^{\prime} = \sqrt{2}\bm{Z}$ scales the physicists' polynomials to probabilists' polynomials.

By Lemma~\ref{lem:CohBounds} these are satisfied whenever
\begin{align*}
\mbox{erfc}(r_p)&<\frac{1}{NP};\\
(1+\delta_p)\frac{\mbox{erfc}\left(\sqrt{(1-2\eta_p(r_p))r_p^2}\right)}{\sqrt{1-2\eta_p(r_p)}}& < \frac{1}{20P^{3/2}}.
\end{align*}
Noting that $\delta_p\rightarrow 0$, and $\mbox{erfc}(r_p)=O(e^{-r_p^2}/r_p)=O(\exp(-(2+\epsilon_p)p)/\sqrt{(2+\epsilon_p)p})$,~\cite{NISTFunctionHandbook}, it follows that the first inequality is satisfied for $r_p=\sqrt{(2+\epsilon_p)p+1}$ if $NP = o(\sqrt{(2+\epsilon_p)p}\exp((2+\epsilon_p)p))$. Recall that we assume that $N=O(P^k)$, and it remains to show that $P^k = o(\sqrt{p}\exp((2+\epsilon_p)p)$, which we address shortly.

From the first point of Lemma~\ref{lem:PolyBeh}, we see for large $p$ that $1-2\eta_p(\sqrt{(2+\epsilon_p)p+1})\ge c$ for a positive constant $c$. The second inequality is then satisfied for $r_p$ if $P^{3/2} = o(\sqrt{p}\exp(c_\epsilon p))$ for an appropriate constant $c_\epsilon>0$ depending on $\epsilon_p$. 

It remains to insure that both of these bounds allow $\epsilon_p$ to go to zero. If $d$ is fixed this holds as $P=O(p^d)= o(\exp(\delta p))$ for any $\delta>0$ establishing the bounds for both conditions for a fixed $d$. We consider the case where $d = c\cdot p$ for some $c>0$. Using Stirling's approximation we have that
\begin{align*}
P &= \frac{(p+d)!}{p!d!}=\frac{((c+1)p)!}{p!(cp)!},\\
 &\approx \sqrt{\frac{c+1}{pc2\pi}}(c+1)^p\left(1+\frac{1}{c}\right)^{cp},
\end{align*}
where the approximation holds with arbitrarily high accuracy as $p,d\rightarrow\infty$. This gives us that for large $p$, 
\begin{align*}
P &\approx \sqrt{\frac{c+1}{pc2\pi}}\beta^p,
\end{align*}
 where $\beta = (c+1)^{c+1}/c^c$. Note that $\beta$ goes to $1$ as $c\rightarrow 0$. It follows in the limit that 
\begin {align*}
P^k &\approx  \left(\frac{\sqrt{c+1}}{\sqrt{cp2\pi}}\right)^k\exp(\alpha k p),
\end{align*}
where $\alpha = \log(\beta)\rightarrow 0$. As $c\cdot p=d > 0$, it follows that $P^k = o(\sqrt{p}\exp(\delta p))$ for any fixed $k,\delta > 0$, establishing both inequalities  needed for the conditions of (\ref{Eqn:CoherenceUnbounded}) when $d=o(p)$ and $N=O(P^k)$.

Having shown $\mathcal{S}$ is acceptable, we now bound $\mu(\bm{\Xi})$ with this choice of $\mathcal{S}$. By Lemma~\ref{lem:CohBounds} and the definition of $\eta_k$ therein, together with the bounds in Table~\ref{Tab:Bounds} we have that,
\begin{align*}
\mu(\bm{\Xi})&\le \exp(\eta_p(r_p)r_p^2)^2,\\
&= \exp\left(\left[1-\frac{\log(2)}{2+\epsilon_p}+o(1)\right]\epsilon_p\right)\exp\left(\left[1-\frac{\log(2)}{2+\epsilon_p}+o(1)\right]2p\right).
\end{align*}
Letting
\begin{align*}
C_p:=&\exp\left(\left[1-\frac{\log(2)}{2+\epsilon_p}+o(1)\right]\epsilon_p\right);\\
\eta_p:=&\exp\left(2\left[1-\frac{\log(2)}{2+\epsilon_p}+o(1)\right]\right),
\end{align*}
it follows that
\begin{align*}
\mu(\bm{\Xi})&\le C_p\eta_p^p,
\end{align*}
As $\epsilon_p \rightarrow 0$, it follows that
\begin{align*}
C_p&\rightarrow 1;\\
\eta_p&\rightarrow\exp(2-\log(2))\approx 3.6945.
\end{align*}

\subsection{Proof of Theorem~\ref{thm:TransformedSamples}}
\label{subsec:ProofTransformedHermite}
Here, we consider a transformation such that $\phi_k(\bm{\xi}) = \psi_k(\bm{\xi})/G(\bm{\xi})$, with $G(\bm{\xi})>0$ so that $|\phi_k(\bm{\xi})| \le C$ uniformly in $k$ and $\bm{\xi}$ for some constant $C$. We may then use that $\psi_k(\bm{\xi}) = \phi_k(\bm{\xi})G(\bm{\xi})$ to identify $\psi_k(\bm{\xi})$ and satisfy the conditions of (\ref{Eqn:CoherenceUnbounded}). We note that this approach corresponds to a weight function $w(\bm{\xi}) = 1/G(\bm{\xi})$. In this framework, we sample $\{\psi_k(\bm{\xi})\}$ from a distribution proportional to $f(\bm{\xi})G^2(\bm{\xi})$, where $f(\bm{\xi})$ is the distribution with respect to which the $\psi_k(\bm{\xi})$ are orthogonal, and use that the $\{\phi_k(\bm{\xi})\}$ form a bounded and approximately orthogonal system. 

By (\ref{eqn:CohBound2}) of Lemma~\ref{lem:CohBounds} and the bounds in Table~\ref{Tab:Bounds}, for $\|\bm{\xi}\|_2\le \sqrt{(2+\epsilon_p)p+1}$ and $k\le p$,
\begin{align}
\label{eqn:hermite_function_bound}
|\psi_k(\bm{\xi})\exp(-\|\bm{\xi}\|_2^2/2)|&\le C,
\end{align}
which suggests taking $G(\bm{\xi})=\exp(\|\bm{\xi}\|^2/2)$. \\

\noindent{\bf Remark.} Notice that the function in the left side of the inequality in (\ref{eqn:hermite_function_bound}) is referred to as Hermite function whose upper bound $C$ is explicitly known, for instance, from \cite{Abramowitz10}. \\

From the argument in the proof of Theorem~\ref{thm:NatSampleCoherence} for the case of one dimensional polynomials,
\begin{align*}
\left|\frac{1}{\sqrt{\pi}}\int_{|\xi|\le \sqrt{(2+\epsilon_p)p+1}}(\psi_i(\xi)\exp(-
\xi^2/2))(\psi_j(\xi))\exp(-\xi^2/2))d\xi -\delta_{i,j}\right|&\le \epsilon_{i,j},
\end{align*}
where $\epsilon_{i,j}$ is small enough to insure that the conditions of (\ref{Eqn:CoherenceUnbounded}) hold. Considering a corresponding change for multidimensional polynomials, let
\begin{align*}
\phi_k(\bm{\xi})&=\pi^{-d/4}\exp(-\|\bm{\xi}\|_2^2/2) V^{1/2}\left(\sqrt{(2+\epsilon_p)p+1},d\right)\psi_{k}(\bm{\xi}),
\end{align*}
where $V(r,d)=(r\sqrt{\pi})^d/\Gamma(d/2+1)$ represents the volume of a $d$-dimensional ball of radius $r$.

If we instead consider a draw from the uniform distribution on the ball of radius $\sqrt{(2+\epsilon_p)p+1}$, then for a $d$-dimensional $\bm{\xi}$,
\begin{align*}
\left|\int_{\|\bm{\xi}\|_2\le \sqrt{(2+\epsilon_p)p+1}}\frac{\phi_i(\bm{\xi})\phi_j(\bm{\xi})}{V(\sqrt{(2+\epsilon_p)p+1},d)}d\bm{\xi} -\delta_{i,j}\right|&\le \epsilon_{i,j},
\end{align*}
and we have that $|\phi_k(\bm{\xi})|$ is bounded, and of order $\pi^{-d/4}V^{1/2}(\sqrt{(2+\epsilon_p)p+1},d)$, giving a bound on the coherence parameter of order $\pi^{-d/2}V(\sqrt{2p},d)$.

\subsection{Key Legendre Lemma}
\label{subsec:keyLegendre}
A key technical simplification is present when working with Legendre polynomials, namely we may fix $\mathcal{S}$ to be $[-1,1]^d$ as a finite number of polynomials on a bounded domain are necessarily bounded. The technical results we require are presented in the following Lemma.

\begin{lem}
\label{lem:Legendre}
For the $1$-dimensional Legendre polynomials,
\begin{align}
\label{eqn:JacBoundChebNat}
\sup_{\xi\in[-1,1]}|\psi_k(\xi)|&=\sqrt{2k+1}.
\end{align}
Further,
\begin{align}
\label{eqn:JacBoundChebTrans}
\sup_{\xi\in[-1,1]}\sqrt{\pi}(1-\xi^2)^{1/4}|\psi_k(\xi)|&\le\sqrt{\frac{2k+1}{k}} \le \sqrt{3}.
\end{align}
\end{lem}
\begin{proof}
These are classical results, with (\ref{eqn:JacBoundChebNat}) following from Theorem 7.32.1 of~\cite{Szego}. We note that a direct application of these theorems does require normalizing the polynomials to be orthonormal. Similarly, (\ref{eqn:JacBoundChebTrans}) follows from Theorem 7.3.3 of~\cite{Szego} and is a direct restatement of Lemma 5.1 of~\cite{RauhutWard}.
\end{proof}

\subsection{Proof of Theorems~\ref{thm:NatSampleCoherenceLeg} and~\ref{thm:TransformedSamplesLeg}}
\label{subsec:JacProofNatSamp}
To show Theorem~\ref{thm:NatSampleCoherenceLeg}, we note that when $p\le d$, (\ref{eqn:LegNatBound}) follows from
\begin{align*}
\mu(\bm{Y})&\le\mathop{\max}\limits_{\|\bm{k}\|_1\le p}\mathop{\prod}\limits_{i=1}^d\|\psi_{k_i}\|^2_{\infty}\\
&\le 3^p \le \exp(2p),
\end{align*}
where we note that at most $p$ of the $d$ dimensions can be non-constant polynomials. Similarly, when $p>d$, 
\begin{align*}
\mu(\bm{Y})&\le\mathop{\max}\limits_{\|\bm{k}\|_1\le p}\mathop{\prod}\limits_{i=1}^d\|\psi_{k_i}\|^2_{\infty}\\
&\le \left(\frac{2p}{d}+1\right)^{d}\\
&\le \exp(2p),
\end{align*}
where the third bound is loose for small $d$.

To show (\ref{eqn:LegTransformedSample}), note that (\ref{eqn:JacBoundChebTrans}) implies that when sampling from the Chebyshev distribution and independently of $p$
\begin{align*}
\mu(\bm{Y})&\le\mathop{\max}\limits_{\|\bm{k}\|_1\le p}\mathop{\prod}\limits_{i=1}^d\|\psi_{k_i}\|^2_{\infty}\\
&\le \mathop{\prod}\limits_{i=1}^d\frac{2k_i+1}{k_i}\le 3^d.
\end{align*}

\subsection{Proof of Theorem~\ref{thm:LowPTransformedSamples}}
\label{subsec:MCMCProof}
The proof of Theorem~\ref{thm:LowPTransformedSamples} follows from a similar logic to the other proofs, but is approachable in a more general measure theoretic setting. By the definition of $B(\bm{\xi})$ in (\ref{eqn:btspec}), we have for all $\bm{\xi}\in\mathcal{S}$ that
\begin{align*}
\mathop{\sup}\limits_{k=1:P}\frac{|\psi_{k}(\bm{\xi})|}{B(\bm{\xi})} = 1.
\end{align*}
This shows that sampling $\bm{Y}$ according to $B(\bm{\xi})$, gives a $\mu(\bm{Y})$ which is achieved uniformly over all values of $\bm{\xi}$. Let
\begin{align*}
c=\left(\int_{\mathcal{S}}f(\bm{\xi})B^2(\bm{\xi})d\bm{\xi}\right)^{-1/2};
\end{align*}
that is, $c^2$ normalizes $f(\bm{\xi})B^2(\bm{\xi})$ to a probability distribution on $\mathcal{S}$. Then for $i,j=1:P$,
\begin{align*}
\int_{\mathcal{S}}\frac{\psi_{i}(\bm{\xi})}{cB(\bm{\xi})}\frac{\psi_{j}(\bm{\xi})}{cB(\bm{\xi})}c^2f(\bm{\xi})B^2(\bm{\xi})d\bm{\xi} =\int_{\mathcal{S}}\psi_{i}(\bm{\xi})\psi_{j}(\bm{\xi})f(\bm{\xi})d\bm{\xi} \approx \delta_{i,j},
\end{align*}
and we assume that $\mathcal{S}$ is chosen so that the approximation holds within the satisfaction of requirements in (\ref{Eqn:CoherenceUnbounded}). As 
\begin{align}
\label{eqn:boundBt}
\mathop{\sup}\limits_{k=1:P}\frac{|\psi_{k}(\bm{\xi})|}{cB(\bm{\xi})} = c^{-1},
\end{align}
for all $\bm{\xi}\in\mathcal{S}$, it follows that the coherence parameter for the scheme associated with sampling from the distribution $c^2f(\bm{\xi})B^2(\bm{\xi})$ and $\bm{\xi}\in\mathcal{S}$ is $c^{-2}$.

We define the measure $\nu$ on Lebesgue measurable subsets of $\mathcal{S}$, denoted by $\mathcal{A}$, via
\begin{align*}
\nu(\mathcal{A}):=\int_{\mathcal{A}}f(\bm{\xi})d\lambda(\bm{\xi}),
\end{align*}
where $\lambda(\bm{\xi})$ is the Lebesgue measure, and $f$ is the distribution with respect to which the $\{\psi_{k}(\bm{\xi})\}$ are orthogonal. Let $\hat{B}$ be a function differing from $B$ on a set of non-zero $\nu$-measure, so that the sampling scheme corresponding to $\hat{B}$ differs on a set of non-zero $\nu$-measure. By (\ref{Eqn:CoherenceUnbounded}), no subset $\mathcal{S}_s$ of $\mathcal{S}$ with $\mu(\mathcal{S}_s)<\mu(\mathcal{S})$ satisfies the conditions of (\ref{Eqn:CoherenceUnbounded}), implying that $\hat{B}$ may not be infinite (corresponding to applying a weight of zero) on any set of positive measure and still satisfy these conditions. As (\ref{eqn:boundBt}) is achieved for all $\bm{\xi}\in\mathcal{S}$, it follows that for the sampling scheme associated with $\hat{B}$, there is a set $\mathcal{A}_{\star}$ with $\nu(\mathcal{A}_{\star})>0$ such that
\begin{align*}
\int_{\mathcal{A}_{\star}}\mathop{\sup}\limits_{k=1:P}\frac{|\psi_{k}(\bm{\xi})|}{\hat{c}\hat{B}(\bm{\xi})}d\lambda(\bm{\xi})> \lambda(\mathcal{A}_{\star})c^{-1},
\end{align*}
and it follows by the Mean Value Theorem for integrals that
\begin{align*}
\mathop{\sup}\limits_{\bm{\xi}\in\mathcal{A}_{\star}}\mathop{\sup}\limits_{k=1:P}\frac{|\psi_{k}(\bm{\xi})|}{\hat{c}\hat{B}(\bm{\xi})}>c^{-1}.
\end{align*}
This implies that,
\begin{align*}
\mathop{\sup}\limits_{\bm{\xi}\in\mathcal{S}}\mathop{\sup}\limits_{k=1:P}\frac{|\psi_{k}(\bm{\xi})|}{\hat{c}\hat{B}(\bm{\xi})}\ge\mathop{\sup}\limits_{\bm{\xi}\in\mathcal{A}_{\star}}\mathop{\sup}\limits_{k=1:P}\frac{|\psi_{k}(\bm{\xi})|}{\hat{c}\hat{B}(\bm{\xi})}>c^{-1}.
\end{align*}
It follows by the definition of $\mu(\bm{Y})$ given in (\ref{Eqn:CoherenceUnbounded}) for the selected $\mathcal{S}$, the coherence parameter $\mu(\bm{Y})$ for the sampling scheme associated with $\hat{B}$ is larger than $c^{-2}$.

\section{Conclusions}
\label{sec:conc}
We provided an analysis of Hermite and Legendre polynomials which allowed us to bound a coherence parameter and generate recovery guarantees for sparse polynomial chaos expansions obtained via $\ell_1$-minimization. We also identified alternative random sampling schemes which provide sharper guarantees and demonstrate improved polynomial chaos reconstructions relative to the random sampling from the orthogonality measure of these bases. These sampling methods were derived based on the properties of Hermite and Legendre polynomials. Furthermore, we showed a Markov Chain Monte Carlo method for generating samples that minimize the coherence parameter, thereby achieving an optimality for the number of random solution realizations. Such a sampling was referred to as coherence-optimal sampling. 

The sampling methods were compared on arbitrary manufactured stochastic functions, and the different sampling strategies were tested for identifying the solution of a 20-dimensional elliptic boundary value problem, where positive results were attained for the coherence-optimal sampling method. Similarly positive results were observed when computing the solution to a non-linear ordinary differential equation, where a high order Hermite polynomial chaos expansion was needed for an accurate solution approximation.

\section*{Acknowledgements}
\label{sec:acknow}

{\color{black}The authors would like to gratefully thank Prof. Akil Narayan (UMass Dartmouth) for his constructive feedback on this manuscript.}

This material is based upon work supported by the U.S. Department of Energy Office of Science, Office of Advanced Scientific Computing Research, under Award Number DE-SC0006402, as well as National Science Foundation under grants DMS-1228359 and CMMI-1201207.

\section*{\texorpdfstring{Appendix A: Proof of Lemma~\ref{lem:PolyBeh}}{Appendix: Proof of Lemma~\ref{lem:PolyBeh}}}
\label{sec:App}

\begin{proof} 
We may rewrite (\ref{eqn:ExpApproximation}) as
\begin{align}
\label{eqn:ExpApp2}
\eta_k(\xi)&=\frac{1}{2}-\frac{\sigma_k}{2\xi}-\frac{\log(C_k)}{\xi^2}-\frac{k}{2\xi^2}+\frac{k\log(\sigma_k +\xi)}{\xi^2}+ \frac{\log\left(\frac{1}{2}\left(1+\frac{\xi}{\sigma_k}\right)\right)}{2\xi^2}.
\end{align}
A straightforward, but lengthy algebraic substitution for $\xi$ yields that for any $\epsilon>0$, as $k\rightarrow\infty$,
\begin{align}
\label{eqn:exponentialConstantApproximation}
\eta_k(\sqrt{(2+\epsilon)k})=\frac{1}{2}-\frac{\log(2)}{2(2+\epsilon)}+o(1).
\end{align}

To show the second point of the Lemma note that $\sigma_k=\sqrt{\xi^2-2k}$ implies that $\partial \sigma_k/\partial \xi=\xi/\sigma_k$, and differentiating the expression (\ref{eqn:ExpApp2}) with respect to $\xi$ gives
\begin{align}
\label{eqn:ExpDerivative}
\frac{\partial\eta_k(\xi)}{\partial \xi} =&\frac{\sigma_k}{\xi^2}+\frac{2\log(C_k)}{\xi^3}+\frac{k}{\xi^3}+\frac{k}{\xi^2\sigma_k}\\
\nonumber
&-\left(\frac{1}{2\sigma_k}+\frac{2k\log(\sigma_k+\xi)}{\xi^3}+\frac{\log\left(\frac{1}{2}\left(1+\frac{\xi}{\sigma_k}\right)\right)}{\xi^3}+\frac{\xi^2-\sigma_k^2}{2\xi^2\sigma_k^2(\sigma_k+\xi)}\right).
\end{align}
Using that $\sigma_k^2=\xi^2-2k$, the above may be rewritten as
\begin{align*}
&\left(2\sigma_k^2\xi^3\right)\frac{\partial\eta_k(\xi)}{\partial \xi} = \xi^2\left(2k+4\log(C_k)-1\right)+\xi\sigma_k\\
&-\left(2\sigma_k^2\left(2k\log(\sigma_k+\xi) +\log\left(\frac{1}{2}\left(1+\frac{\xi}{\sigma_k}\right)\right)\right)+4k(k+2\log(C_k))\right).
\end{align*}
It follows that $\partial\eta_k(\xi)/\partial \xi<0$ whenever
\begin{align*}
2\sigma_k^2&\left(2k\log(\sigma_k+\xi) +\log\left(\frac{1}{2}\left(1+\frac{\xi}{\sigma_k}\right)\right)\right)+4k(k+2\log(C_k))\\
 &> \xi^2\left(2k+4\log(C_k)-1\right)+\xi\sigma_k.
\end{align*}
Substituting $\sigma_k^2$ for $\xi$ and $k$, gives that this condition is equivalent to
\begin{align*}
\frac{\xi^2}{2k}&\left(2k\log(\sigma_k+\xi) +\log\left(\frac{1}{2}\left(1+\frac{\xi}{\sigma_k}\right)\right)-k-2\log(C_k)+\frac{1}{2}-\frac{\sigma_k}{2\xi}\right)\\
 &>\left(2k\log(\sigma_k+\xi)+\log\left(\frac{1}{2}\left(1+\frac{\xi}{\sigma_k}\right)\right)-k-2\log(C_k)\right).
\end{align*}
From this form and using that $\sigma_k/\xi<1$, it follows that the derivative is negative for large enough $\xi$. More precisely, let
\begin{align*}
X_\xi &:= 2k\log(\sigma_k+\xi)+\log\left(\frac{1}{2}\left(1+\frac{\xi}{\sigma_k}\right)\right)-k-2\log(C_k);\\
Y_\xi &:= \frac{1}{2}-\frac{\sigma_k}{2\xi};\\
Z_\xi &:= \frac{\xi^2}{2k},
\end{align*}
where we wish to identify $\xi$ such that $Z_\xi(X_\xi+Y_\xi)>X_\xi$, which is equivalent to $Z_\xi Y_\xi>(1-Z_\xi)X_\xi$. Let $\epsilon_k\ge 0$, and note that if $\xi\ge\sqrt{(2+\epsilon_k)k+1}$ then $\sigma_k/\xi<1$ which implies that $Y_\xi>0$. Further, for $\xi\ge\sqrt{(2+\epsilon_k)k+1}$, it follows that $Z_\xi>1$. We now identify an $\epsilon_k$ such that for $\xi\ge\sqrt{(2+\epsilon_k)k+1}$, we verify that $X_\xi>0$, from which it follows that $Z_\xi Y_\xi>(1-Z_\xi)X_\xi$. Note that $\epsilon_k\ge 0$, implies that $\sigma_k \ge 1$ and as such for $\xi\ge\sqrt{(2+\epsilon_k)k+1}$,
\begin{align*}
X_\xi&\ge 2k\log\left(\sqrt{(2+\epsilon_k)k}\right)-k-2\log(C_k).
\end{align*}
Recalling that $C_k = \sqrt{2^kk!}$, we conclude from properties of the Log Gamma function~\cite{LogGammaAnalysisBook} that
\begin{align*}
\log(k!)&=(k+1)\log(k+1)-(k+1)-\frac{1}{2}\log\left(\frac{k}{2\pi}\right) + O(k^{-1});\\
2\log(C_k)&=k\log(2)+(k+1)\log(k+1)-(k+1)-\frac{1}{2}\log\left(\frac{k}{2\pi}\right) + O(k^{-1}).
\end{align*}
From this we may simplify terms, leading to a lower bound on $X_\xi$ for some $C>0$ given by
\begin{align*}
X_\xi&\ge k\log\left(\frac{(2+\epsilon_k)k}{2(k+1)}\right)+\left(1-\frac{\log(2\pi)}{2}\right) - \log\left(\frac{k+1}{\sqrt{k}}\right) -\frac{C}{k}.
\end{align*}
Noting that $1>\log(2\pi)/2$, it follows that we may guarantee that $X_\xi$ is positive for some sequence of $\epsilon_k>0$ which admits that $\epsilon_k\rightarrow 0$. It follows that we have a monotonic derivative for $\eta_k(\xi)$ with respect to $\xi$ for $\xi\ge\sqrt{(2+\epsilon_k)k}$, and thus conclude the second point of this Lemma.

To show the third point, we utilize a differential-difference equation~\cite{Szego} for orthonormal Hermite polynomials,
\begin{align}
\label{eqn:HermDiff}
\sqrt{2(k+1)}\psi_{k+1}(\xi)&=2\xi\psi_k(\xi)-\psi_k^{\prime}(\xi).
\end{align}
We note here that the approximation in~\cite{Hermite1} giving the approximation of (\ref{eqn:ExpApproximation}) for $\psi_k$ extends to the derivative $\psi_k^{\prime}$ so that for sufficiently large $k$ the approximation to $\psi_k$ in (\ref{eqn:ExpApproximation}) is arbitrarily accurate for $\psi_k$, and when differentiated to $\psi^{\prime}_k$. Differentiating with the use of the chain rule gives that
\begin{align}
\label{eqn:HermChain}
\frac{\partial}{\partial \xi} \exp(\eta_k(\xi) \xi^2) = \exp(\eta_k(\xi) \xi^2)\left(\xi^2\frac{\partial \eta_k(\xi)}{\partial \xi} + 2\xi\eta_k(\xi)\right).
\end{align}
Plugging (\ref{eqn:ExpApp2}) and (\ref{eqn:ExpDerivative}) into (\ref{eqn:HermChain}), and in turn into (\ref{eqn:HermDiff}) we have that
\begin{align*}
\psi_{k+1}(\xi)&\approx \frac{2\xi\exp(\eta_k \xi^2)-\frac{\partial}{\partial \xi} \exp(\eta_k \xi^2)}{\sqrt{2(k+1)}},\\
& = \frac{\exp(\eta_k \xi^2)}{\sqrt{2(k+1)}}\left(\xi\left(1+\frac{\xi}{2\sigma_k}\right)+\frac{\xi^2-\sigma_k^2}{2\sigma_k^2(\sigma_k+\xi)}\right).
\end{align*}
It follows by (\ref{eqn:ExpApproximation}) that
\begin{align}
\nonumber
\frac{\psi_{k+1}(\xi)}{\psi_{k}(\xi)}&\approx\frac{\exp(\eta_{k+1}(\xi)\xi^2)}{\exp(\eta_k(\xi)\xi^2)},\\
\label{eqn:HermiteRatio}
&\approx\frac{1}{\sqrt{2(k+1)}}\left(\xi\left(1+\frac{\xi}{2\sigma_k}\right)+\frac{\xi^2-\sigma_k^2}{2\sigma_k^2(\sigma_k+\xi)}\right),
\end{align}
Letting $k$ go to infinity it follows that for some $\epsilon_k\rightarrow 0$ and $|\xi|>\sqrt{(2+\epsilon_{k+1})(k+1)+1}$, both the function and derivative approximations considered of $\exp(\eta_k(\xi)\xi^2)$ to $\psi_k(\xi)$ and $\partial \exp(\eta_k(\xi)\xi^2)/\partial$ to $\psi^{\prime}_k(\xi)$ are arbitrarily accurate~\cite{Hermite1}. If the right-hand side of (\ref{eqn:HermiteRatio}) is larger than 1, and $k$ is large enough so that the approximation is sufficiently accurate for $|\xi|=\sqrt{(2+\epsilon_{k+1})(k+1)+1}$, then we will show that (\ref{eqn:HermiteRatio}) implies that $\eta_{k+1}(\xi)>\eta_k(\xi)$.

We note that larger $\epsilon_k,\epsilon_{k+1}$ complicate the proof, but small enough $\epsilon_k,\epsilon_{k+1}$ do not affect the comparisons made, and thus for brevity of presentation we set $\epsilon_k=\epsilon_{k+1}=0$ for the remainder of this proof. Note that the ratio $\psi_{k+1}(\xi)/\psi_{k}(\xi)$ is monotonically increasing for $|\xi|>\sqrt{2(k+1)+1}$ as $\partial \sigma_k/\partial \xi < 1$, and so it suffices to check when $|\xi|=\sqrt{2(k+1)+1},$ and thus $\sigma_k=\sqrt{3}$. In this case the ratio in (\ref{eqn:HermiteRatio}) satisfies
\begin{align}
\label{eqn:HermGap}
\sqrt{\frac{2k+3}{2k+2}}+\sqrt{\frac{3}{2(k+1)}}+\frac{k/3}{\sqrt{2k+2}(\sqrt{2k+3}+\sqrt{3})}+\frac{k\sqrt{2}/\sqrt{3}}{\sqrt{k+1}}&\approx\frac{\psi_{k+1}(\xi)}{\psi_{k}(\xi)},\\
&>1,
\end{align}
as the first term is always larger than $1$ and all terms are positive, with the last term increasing in $k$. It follows that the lemma is established when both $k_1$ and $k_0$ are larger than some $K$ which insures that both the approximation in (\ref{eqn:ExpApproximation}) is sufficiently accurate and that the gap in (\ref{eqn:HermGap}) is sufficiently large.

For $k_0<K$ note that $\sigma_{k_0}(\sqrt{2k_1+1})=\sqrt{2(k_1-k_0)+1}$ satisfies 
\begin{align*}
\frac{\sigma_{k_0}(\sqrt{2k_1+1})}{\sqrt{2k_1+1}}&=\frac{\sqrt{2(k_1-k_0)+1}}{\sqrt{2k_1+1}}\rightarrow 1,
\end{align*}
as $k_1/k_0 \rightarrow \infty$, while $\sigma_{k_1}(\sqrt{2k_1+1})=1$ remains fixed. From the expression for $\eta_k(\xi)$ in (\ref{eqn:ExpApp2}) it follows for any $k_0$, a large enough $K$, and $k_1\ge K$ that $\eta_{k_0}(\sqrt{2k_1+1})<\eta_{k_1}(\sqrt{2k_1+1})$, showing the third point of this Lemma.
\end{proof}

\bibliographystyle{elsarticle-num}
\bibliography{CoherencePolyCitations}
\end{document}